\documentclass{amsart}

\usepackage{geometry}                
\usepackage[applemac]{inputenc}
\usepackage[T1]{fontenc}
\usepackage{indentfirst}
\usepackage{diagmac2}
\usepackage{lipsum}
\usepackage{csquotes}
\usepackage{emptypage}
\usepackage{faktor} 
\usepackage{amsmath,amsfonts,amssymb,amsthm}
\usepackage{latexsym}
\usepackage{mathdots}
\usepackage{mathrsfs}  
\usepackage{comment}
\usepackage[T1]{fontenc}
\usepackage[hidelinks]{hyperref}
\usepackage{tikz-cd}
\usepackage{tikz}
\usepackage{xcolor}

\definecolor{lime}{HTML}{A6CE39}
\DeclareRobustCommand{\orcidicon}{%
	\begin{tikzpicture}
	\draw[lime, fill=lime] (0,0) 
	circle [radius=0.16] 
	node[white] {{\fontfamily{qag}\selectfont \tiny ID}};
	\draw[white, fill=white] (-0.0625,0.095) 
	circle [radius=0.007];
	\end{tikzpicture}
	\hspace{-2mm}
}

\foreach \x in {A, ..., Z}{%
	\expandafter\xdef\csname orcid\x\endcsname{\noexpand\href{https://orcid.org/\csname orcidauthor\x\endcsname}{\noexpand\orcidicon}}
}

\newcommand{\Q}{\mathbb{Q}}
\newcommand{\R}{\mathbb{R}}
\newcommand{\F}{\mathbb{F}}
\newcommand{\K}{\mathbb{K}}

\DeclareMathOperator{\Ima}{Im}
\DeclareMathOperator{\ig}{ig}

\DeclareMathOperator{\Id}{id}

\theoremstyle{plain}
\newtheorem{theorem}{Theorem}[section]
\newtheorem{proposition}[theorem]{Proposition}
\newtheorem{lemma}[theorem]{Lemma}
\newtheorem{corollary}[theorem]{Corollary}
\newtheorem{fact}[theorem]{Fact}
\theoremstyle{definition}
\newtheorem{definition}[theorem]{Definition}
\newtheorem{example}[theorem]{Example}

\theoremstyle{remark}
\newtheorem{remark}[theorem]{Remark}

\begin{document}

\title{Notes on valuation theory for Krasner hyperfields}

\author[Linzi, A.]{Alessandro Linzi\orcidA{}}

\address{Center for Information Technologies and Applied Mathematics, University of Nova Gorica, Slovenia.}

\email{alessandro.linzi@ung.si}

\thanks{The author would like to spend a few words to thank H.\ Stoja\l owska, P.\ Touchard, Franz-Viktor and Katarzyna Kuhlmann, I.\ Cristea as well as Ch.\ Massouros, who all, in one way or another, contributed to the realisation of the final version of this manuscript.}

\subjclass{Primary: 12J20, 20N20 Secondary: 13A18.}

\keywords{Hyperfield, multifield, hyperring, valuation, ordered abelian group, tropical hyperfield}


\begin{abstract}
The main aim of this article is to study and develop valuation theory for Krasner hyperfields. In analogy with classical valuation theory for fields, we generalise the formalism of valuation rings to describe equivalence of valuations on hyperfields. After proving basic results and discussing several examples, we focus on the valued hyperfields that Krasner originally defined in 1957. We find that these must have a particular additive structure which in turns implies the existence of a valuation a'la Krasner. We note that given such a valued hyperfield $(F,v)$, the valuation induced by its additive structure does not have to be equivalent to $v$. We discuss the cases in which it does. 
\end{abstract}

\maketitle

\section{Introduction}

In 1957, M.\ Krasner in \cite{Kra57} (the article is included in Krasner's collected works \cite[pages 413--490]{BP19}) formulated for the first time an axiomatisation of structures that generalise fields by allowing the additive operation to be multivalued (see \cite[pages 407--490]{BP19} for more details). He called these hyperfields (in french \textit{hypercorps}).

On the one hand, he found his inspiration in the work of Marty \cite{Mar34,Mar35,Mar36} which is considered as the starting point of hypergroup theory. In fact, the additive part of a hyperfield is a hypergroup (see also \cite{Vuk16,Mas21} for a more detailed historical overview).

On the other hand, Krasner was motivated by his interest in valuation theory (in connection with $p$-adic numbers) and he sensed the importance of some hyperfields which are canonically associated to any valued field (cf.\ Example~\ref{Kgamma}). Let us briefly recall some basic notions of classical valuation theory.

Let $K$ be a field and $\Gamma$ a linearly ordered abelian group (written additively). A surjective map
\[
v:K\to\Gamma\cup\{\infty\}
\]
is called a \emph{(Krull) valuation on }$K$ if it satisfies for all $x,y\in K$:
\begin{itemize}
\item[--] $v(x)=\infty$ if and only if $x=0$,
\item[--] $v(xy)=v(x)+v(y)$,
\item[--] $v(x+y)\geq\min\{v(x),v(y)\}$.
\end{itemize}
Here, $\infty$ is a symbol such that $\gamma+\infty=\infty+\gamma=\infty>\gamma$ for all $\gamma\in\Gamma$. If a valuation $v$ on a field $K$ is given, then $(K,v)$ is called a \emph{valued field}. One usually denotes $\Gamma$ by $vK$ and call it the \emph{value group of }$(K,v)$. The \emph{value} $v(x)$ of $x\in K$ will often be written as $vx$, if there is no risk of confusion.

If $(K,v)$ is a valued field, then
\[
\mathcal{O}_v:=\{x\in K\mid vx\geq 0\}
\]
is a subring of $K$, called the \emph{valuation ring} of $(K,v)$. It determines the valuation map $v$ \emph{up to equivalence}, i.e., up to composition with an order preserving isomorphism of the value group. Any valuation ring has a unique maximal ideal
\[
\mathcal{M}_v:=\{x\in K\mid vx> 0\}
\]
and the field $Kv:=\mathcal{O}_v/\mathcal{M}_v$ is called the \emph{residue field} of $(K,v)$. For further details, let us mention \cite{PE05,PR84} as general references on classical valuation theory.

While it is quite natural to generalise the concept of valuation to hyperfields (see Definition \ref{hyperval}), Krasner noticed that the structures which attracted his attention come equipped with a map similar to a valuation which satisfies two additional properties (see Definition \ref{KVH}). These properties would be vacuous if postulated for valued fields. Thus, Krasner included them in his axiomatisation of valued hyperfields (\textit{hypercorps valu\'e}). 

Krasner's motivation was not model theoretical. Nevertheless, it turns out that the structures he studied play an important role in the model theory of valued fields; specifically for the problem of quantifier elimination for henselian valued fields (of characteristic $0$). In this setting, these objects are known as RV-structures or leading-term structures and have been considered, independently of Krasner, by Flenner in \cite{Fle11} (see also \cite{Bas91,Kuh94,PhD22,LT22} for more details). In addition, an interesting application of hyperfields in the model theory of valued fields recently appeared in \cite{Lee20}.

The interest for valued hyperfields may also be motivated by the following observations. Classically, real algebra, which studies real fields (i.e., linearly ordered fields, with an order compatible with the operations) and was developed by E.\ Artin in his solution of Hilbert's seventeenth problem, is in relation with valuation theory. After the works of Marshall and G\l adki \cite{GM12,GM17} on real hyperfields, it is to expect that a development of valuation theory for hyperfields would be beneficial for this line of research. Significant developments have already been achieved with the generalisation of classical results, such as the Baer--Krull Theorem, to the multivalued setting (see \cite{Kru32,KLS22}). The reader interested in these aspects may look also at \cite{Gla10,Gla17,GlaNotes}. Furthermore, the unpublished work of Viro \cite{Vir10}, followed up by Jun and Jell et al.\ \cite{Jun21, JSY22}, indicate applications in the realm of tropicalization maps and analytification, topics that are as well, classically, in relation with valuation theory. 

We believe that developing valuation theory for hyperfields will eventually lead to further applications back to the classical theory of valuations for fields.

The switch from singlevalued operations to multivalued ones may be not difficult to describe and understand. Nevertheless, the consequences of this switch have to be handled very carefully. For instance, valuation theorists are accostumed to the fact that $v(x-y)$ always defines a metric (in fact, an ultrametric, cf.\ Section \ref{UMS}) on a valued field $(K,v)$, but this clearly ceases to be true (in general) if the operation of addition (and hence of difference) is multivalued. Indeed, in the latter case, the distance map may depend on the choice of some element of the set $x-y$ and this choice is not canonical, in general. This obstacle is overtaken using one of Krasner's postulates which forces all the elements of $x-y$, for $x\neq y$, to have the same value. However, one may expect the latter being a quite strong requirement. In fact, the restrictions due to Krasner's axioms for valued hyperfields are so strong that, e.g., the trivial valuation on a hyperfield $F$ (Example \ref{extriv}) satisfies them only if $F$ is actually a field (Remark \ref{trivialKrasval}). 
For this and other reasons (which we will discuss later in the paper), it makes sense to consider also valuations on hyperfields which do not satisfy the two additional properties imposed by Krasner (cf.\ Example \ref{noKraVal}) and it turns out that this choice is more beneficial than harmful. As a consequence of these observations, the term \lq\lq valued hyperfield\rq\rq\ will be used in this paper in a broader sense than the original one; the valued hyperfields satisfying in addition Krasner's axioms will be called \lq\lq Krasner valued hyperfields\rq\rq and their valuations \lq\lq Krasner valuations\rq\rq.

The manuscript is organised as follows. In Section \ref{sec1} we introduce the necessary concepts and terminology from hypercompositional algebra, taking the opportunity to discuss several examples. In Section \ref{sec2}, we show that the formalism of valuation rings to describe valuations up to equivalence generalises without major modifications to the multivalued framework. We use this formalism to state and prove one of our main results Theorem \ref{maintw} on Krasner valued hyperfields in Section \ref{sec3}.\par
Our main theorem states that the existence of a Krasner valuation on a hyperfield $F$ implies that the additive structure of $F$ satisfies certain additional axioms (cf.\ Proposition \ref{Mittas1}). Conversely, if the additive structure of a hyperfield $F$ satisfies those axioms, then $F$ admits a Krasner valuation (cf.\ Proposition \ref{Mittas2}). Let $v$ be a Krasner valuation on a hyperfield $F$. Then the additive structure of $F$ induces a Krasner valuation $w$ on $F$. We observe that, in general, $w$ is not equivalent to $v$ (Example \ref{last}). After Section \ref{sec4}, where we put into our context the notion of coarsening of a valuation, which turns out to be necessary for our discussion, we describe situations in which $v$ and $w$ are equivalent in Section \ref{sec5}. The final Section \ref{sec6} of this manuscript contains possible further lines of investigation.

Besides presenting some new results, the article is also thought to serve as a unified reference for basic hyperring and hyperfield theory and related terminology. On the one hand, these hypercompositional structures have recently attracted increasing interest from the mathematical community (including top mathematicians such as A.\ Connes \cite{CC10,CC11}). On the other hand, we found the relevant literature on hyperrings and hyperfields to be quite fragmented. After the time and effort spent during the PhD studies and thanks to the fruitful connections with some of the mathematicians who are part of the early history of hypercompositional algebra, we felt that we could give a contribution from this point of view as well.

\section{Preliminaries}\label{sec1}

Let $H$ be a non-empty set and $\mathcal{P}(H)$ its power set. A \emph{multivalued operation} $+$ on $H$ is a function which associates to every pair $(x,y) \in H \times H$ an element of $\mathcal{P}(H)$, denoted by $x+y$. If $+$ is a multivalued operation on $H\neq\emptyset$, then for $x \in H$ and $A,B\subseteq H$ we set 
\[
A+B:=\bigcup_{a\in A,b\in B} a+b,
\]
$A + x := A + \lbrace x \rbrace$ and $x +A := \lbrace x \rbrace + A$. If $A$ or $B$ is empty, then so is $A+B$.\par
A \emph{hypergroup} can be defined as a non-empty set $H$ with a  multivalued operation $+$ which is associative (see Definition \ref{hypergp} (CH1) below) and \emph{reproductive on $H$} (i.e., $x + H = H + x = H$ for all $x \in H$).
This notion was first considered by F.\ Marty in \cite{Mar34,Mar35,Mar36}. The theory of hypergroups, with a detailed historical overview, is presented in \cite{Mas21}, where an extensive bibliography is also provided.
\begin{definition}
A \emph{hyperoperation} $+$ on $H$  is a multivalued operation such that $x+y\neq\emptyset$ for all $x,y\in H$.
\end{definition} 
\begin{lemma}[Theorem 12 in \cite{Mas21}]\label{x+ynempty}
If $(H,+)$ is a hypergroup, then $+$ is a hyperoperation on $H$. 
\end{lemma}
\begin{proof}
Aiming for a contradiction, suppose that $x+y=\emptyset$ for some $x,y\in H$. Then
\[
H=x+H=x+(y+H)=(x+y)+H=\emptyset+H=\emptyset,
\]
which is excluded.
\end{proof}

The following special class of hypergroups (cf.\ Remark \ref{can} below) will be of interest for us.

\begin{definition} \label{hypergp}
A \emph{canonical hypergroup} is a triple $(H,+,0)$, where $H\neq\emptyset$, $+$ is a multivalued operation on $H$ and $0$ is an element of $H$ such that the following axioms hold:
\begin{itemize}
\item[(CH1)] $+$ is associative, i.e., $(x+y)+z=x+(y+z)$ for all $x,y,z \in H$,
\item[(CH2)] $x+y=y+x$ for all $x,y\in H$,
\item[(CH3)] for every $x\in H$ there exists a unique $x'\in H$ such that $0\in x+x'$ (the element $x'$ will be denoted by $-x$),
\item[(CH4)] $z\in x+y$ implies $y\in z-x:=z+(-x)$ for all $x,y,z\in H$.
\end{itemize}
\end{definition}

\begin{lemma}\label{can}
Let $(H,+,0)$ be a canonical hypergroup. Then the multivalued operation $+$ is reproductive on $H$. In particular, $(H,+)$ is a hypergroup and $+$ is a hyperoperation.
\end{lemma}
\begin{proof}
Fix $a \in H$. For all $x \in H+a$ there exists $y \in H$ such that $x \in y+a \subseteq H$. Therefore, $H+a \subseteq H$. For the other inclusion, observe that for all $x\in H$ we have that 
\[
x \in x+0 \subseteq x+(a-a) = (x-a) + a,
\]
so there exists $y \in x-a \subseteq H$ such that $x \in y+a \subseteq H+a$. We have proved that $H=H+a$ for an arbitrary $a\in H$. Now the conclusions of the lemma follow from (CH2), the definition of hypergroups and Lemma \ref{x+ynempty}.
\end{proof}

Let $(H,+,0)$ be a canonical hypergroup. Then $0$ is called the \emph{neutral element for $+$ in $H$} because of the following observation.

\begin{lemma}[Section III, (b) in \cite{Mas85}]
Let $(H,+,0)$ be a canonical hypergroup. Then $x+0 = \lbrace x \rbrace$ for all $x\in H$.
\end{lemma}

\begin{proof}
Take $x\in H$. If $y\in x+0$, then $0 \in y-x$ follows by (CH4). Now $y = x$ follows from the uniqueness required in (CH3).
\end{proof}
\begin{remark} \label{grhy}
Note that an abelian group $(G,*,0)$ is not a priori a canonical hypergroup, because the operation on $G$ is not a multivalued operation, as it takes values in $G$ and not in $\mathcal P(G)$. However, it can be turned into a canonical hypergroup $(G,+,0)$ by setting $x + y := \lbrace x*y \rbrace$ for all $x,y\in G$. Conversely, if a canonical hypergroup $(H,+,0)$ satisfies that $x+y$ is a singleton for all $x,y\in H$, then one can define on $H$ a standard operation which makes it an abelian group with identity $0$. In these cases, we sometimes abusively say that $(G,*,0)$ is a canonical hypergroup or that $(H,+,0)$ is a group. 
\end{remark}
%
%

\subsection{Hyperrings and hyperfields}

Let us now define the structures of main interest in this paper.

\begin{definition}\label{Krasnerhyperrings}
A \emph{(commutative) hyperring} is a tuple $(R,+,\cdot,0)$ which satisfies the following axioms:
\begin{itemize}
\item[(HR1)] $(R,+,0)$ is a canonical hypergroup,
\item[(HR2)] $(R,\cdot)$ is a (commutative) semigroup and $0$ is an absorbing element, i.e., $0\cdot x=x\cdot0= 0$, for all $x\in R$,
\item[(HR3)] the operation $\cdot$ is distributive with respect to $+$. That is, for all $x,y,z\in R$,
\[
x(y+z)=xy+xz,
\]
\end{itemize}
where, for $x\in R$ and $A\subseteq R$, we have set
\[ 
xA:=\{xa\mid a\in A\}.
\]
If the operation $\cdot$ has a neutral element $1 \neq 0$, then we say that $(R,+,\cdot,0,1)$ is a \emph{hyperring with unity}.
If  $(R,+,\cdot,0,1)$  is a hyperring with unity and 
$(R\setminus\{0\}, \cdot,1)$ is an abelian group, then $(R,+,\cdot,0,1)$ is called a \emph{hyperfield}.
If $R$ is a hyperring with unity, then we denote by $R^\times$ the multiplicative group of the \emph{units} of $R$, i.e.,
\[
R^\times:=\{x\in R\mid\exists y\in R~:~xy=1\}.
\]
In particular, if $R$ is a hyperfield, then $R^\times=R\setminus\{0\}$.
\end{definition}

Since we will only consider the commutative case, in the following sections we will call commutative hyperrings simply \emph{hyperrings}.\par

\begin{remark}\label{prefix}
In the literature, the name \lq\lq hyperring\rq\rq\ can be found for structures $(R,+,\cdot)$ where $+$ is an operation and $\cdot$ is a multivalued operation or where both $+$ and $\cdot$ are multivalued operations. Some authors refer to the structures $(R,+,\cdot,0)$ where only the addition is multivalued as Krasner hyperrings (some more historical remarks on this are given in in \cite[Section 4]{Gol18}). 

Structures $(R,+,\cdot,0)$ where $+$ is an operation and $\cdot$ is a multivalued operation (satisfying similar axioms) were introduced in \cite{Rot82} and called multiplicative hyperrings (in italian \textit{iperanelli moltiplicativi}).

Structures $(R,+,\cdot,0)$ with both $+$ and $\cdot$ multivalued operations (satisfying similar axioms) are instead called general hyperrings (see e.g. \cite[Section 2]{CJ13}).

In this paper, the name \lq\lq hyperring\rq\rq\ will be used for Krasner hyperrings exclusively, as indicated in the above definition.\par

In the literature, one may also find the term multiring referring to a structure $(R,+,\cdot,0)$, where $+$ is a multivalued operation and $\cdot$ is an operation, satisfying (HR1), (HR2) and the following weaker version of (HR3):
\begin{itemize}
\item[(MR)] $x\cdot(y+z)\subseteq x\cdot y+x\cdot z$, for all $x,y,z\in M$.
\end{itemize}
Thus, in this case, the different name reflects a difference in the axioms rather than in the structure.\par 
Similarly, multifields are defined as hyperfields satisfying (MR) instead of (HR3). Multirings and multifields have been considered for instance in \cite{Gla10}, where, among other facts, it is observed that all multifields are hyperfields, while there are several meaningful examples of multirings that are not hyperrings.\par
In this paper, we preferred to use the name \lq\lq hyperfield\rq\rq\ solely.
\end{remark}

We say that a hyperring (resp.\ hyperfield) $R$ is a ring (resp.\ field) if the additive canonical hypergroup of $R$ is a group (cf.\ Remark \ref{grhy}). The next observation gives a necessary condition for a hyperring with unity to be a ring. This fact is an immediate corollary of a result already noted in \cite[page 369]{Mit73} (see also Example \ref{scalars} below). We wish to state it for later reference and we take the opportunity to write a quick proof . 

\begin{lemma}[\cite{Mit73}]\label{1-1=0}
A hyperring with unity $R$ is a ring if and only if $1-1=\{0\}$.
\end{lemma}
\begin{proof}
If $R$ is a field, then $1-1=\{0\}$ holds trivially. Conversely, by axiom (HR3) $1-1=\{0\}$ implies $x-x=\{0\}$ for all $x\in R$. Take $a,b\in R$ and $x,y\in a+b$. We have that
\[
x-y\subseteq (a+b)-(a+b)=(a-a)+(b-b)=\{0\}
\]
In particular, $0\in x-y$ and $x=y$ follows fro (CH3). We have proved that $a+b$ is a singleton for all $a,b\in R$.
\end{proof}

\begin{example}\label{K}
The \emph{$\K$-hyperfield} $\K$ is the set $\{0,1\}$ with the hyperoperation $+$ which has $0$ as its neutral element and satisfies $1+1=\{0,1\}$. The multiplication is the obvious one.
\end{example}

\begin{remark}
It is customary among some authors to call the above hyperfield the Krasner hyperfield. We prefer to avoid that terminology which, in view of Remark \ref{prefix} above, might lead to unnecessary confusion. 
\end{remark}

\begin{example}\label{S}
The \emph{sign hyperfield} $\mathbb{S}$ is the set $\{-1,0,1\}$ with the hyperoperation $+$ which has $0$ as its neutral element and satisfies $x+x=\{x\}$ for $x\in\{-1,1\}$ and $1-1=\{-1,0,1\}$. The multiplication is the obvious one.
\end{example}

\begin{example}\label{W}
The \emph{weak sign hyperfield} $\mathbb{W}$ is the set $\{-1,0,1\}$ with the hyperoperation $+$ which has $0$ as its neutral element and satisfies $x+x=\{x,-x\}$ for $x\in\{-1,1\}$ and $1-1=\{-1,0,1\}$. The multiplication is the obvious one.
\end{example}

Other examples of finite hyperfields are described e.g.\ in \cite{AEH20}. Let us now consider some infinite examples. 

\begin{example}\label{TGamma}
Let $(\Gamma,+,<,0)$ be an ordered abelian group and $\infty$ be a symbol such that $\gamma+\infty=\infty+\gamma=\infty>\gamma$ for all $\gamma\in\Gamma$. For $x,y\in\Gamma\cup\{\infty\}$ such that $x\leq y$ (i.e., $x<y$ or $x=y$), let us denote by $[x,y]$ the set consisting of all $z\in\Gamma\cup\{\infty\}$ such that $x\leq z\leq y$. Consider the hyperoperation $\boxplus$ defined on $\mathcal{T}(\Gamma):=\Gamma\cup\{\infty\}$ as $x\boxplus \infty=\infty\boxplus x=\{x\}$ for all $x\in\mathcal{T}(\Gamma)$ and:
\[
x\boxplus y:=\begin{cases}\{\min\{x,y\}\}&\text{if }x\neq y,\\ [x,\infty]&\text{if }x=y.\end{cases}
\quad\quad(x,y\in\mathcal{T}(\Gamma))
\]
It is not difficult to check that $(\mathcal{T}(\Gamma),\boxplus,+,\infty,0)$ is a hyperfield. The hyperfield $\mathcal{T}(\R,+,0,>)$, where $>$ denotes the standard order of the real numbers, is known as the tropical hyperfield (see \cite[Section 5.3]{Vir10}). Therefore, we call the hyperfields of the form $\mathcal{T}(\Gamma)$ \emph{generalised tropical hyperfields}.
\end{example}

\begin{example}\label{T'Gamma}
Let $(\Gamma,+,<,0)$ be an ordered abelian group and $\infty$ be a symbol such that $\gamma+\infty=\infty+\gamma=\infty>\gamma$ for all $\gamma\in\Gamma$. For $x,y\in\Gamma\cup\{\infty\}$ such that $x<y$, let us denote by $(x,y]$ the set consisting of all $z\in\Gamma\cup\{\infty\}$ such that $x<z\leq y$. Consider the hyperoperation $\boxplus'$ defined on $\mathcal{T}(\Gamma)$ as $x+\infty=\infty+x=\{x\}$ for all $x\in\mathcal{T}(\Gamma)$ and:
\[
x\boxplus' y:=\begin{cases}\{\min\{x,y\}\}&\text{if }x\neq y,\\ (x,\infty]&\text{if }x=y.\end{cases}
\quad\quad(x,y\in\mathcal{T}(\Gamma))
\]
It is not difficult to check that $(\mathcal{T}(\Gamma),\boxplus',+,\infty,0)$ is a hyperfield. We will denote it by $\mathcal{T}'(\Gamma)$ and call it the  \emph{strict generalised tropical hyperfield}.
\end{example}

The above examples are (up to isomorphism) all instances of a general construction which yields a hyperring from a ring and a subgroup of its multiplicative semigroup (see Proposition \ref{exquot} below). This construction was first described by Krasner in \cite{Kra83}. We briefly recall it in the following example.

\begin{example}\label{Kraquot}
Let $A$ be a commutative ring and $T$ a subgroup of the commutative semigroup $(A,\cdot)$. Let $A_T$ denote the set of all multiplicative cosets $xT$ for $x\in A$, in particular $0T=\{0\}$. Krasner showed that by setting
\[
xT+ yT:=\{(x+yt)T\mid t\in T\}
\]
and
\[
xT\cdot yT:=xyT
\]
we have that $(A_T,+,\cdot,0T)$ is a hyperring where $-(xT)=(-x)T$ for all $x\in A$. In addition, if $A$ is a field, then this construction always yields a hyperfield. This is called the \emph{factor (or quotient) hyperring (resp.\ hyperfield)} of $A$ modulo~$T$.
\end{example}

For the next proposition we employ the notion of isomorphism of hyperfields (see Definition \ref{isomorphism} below).

\begin{proposition}\label{exquot}
The following assertions hold:
\begin{itemize}
\item[$(i)$] For any field $K$ with more than two elements, we have that $\K$ is isomorphic to $K_{K^\times}$. 
\item[$(ii)$] For any ordered field $K$ with positive cone $P$, we have that $\mathbb{S}$ is isomorphic to $K_{P}$. 
\item[$(iii)$] If $p>3$ is a prime such that $p\equiv 3\mod 4$, then $\mathbb{W}$ is isomorphic to $(\F_p)_{(\F_p^\times)^2}$.  
\item[$(iv)$] Let $\Gamma$ be an ordered abelian group. For any field $k$ with more than two elements, let $K:=k((t^\Gamma))$ denote the Hahn series field and let $v$ be its canonical $t$-adic valuation. Then $K_{\mathcal{O}^\times_v}$ is isomorphic to $\mathcal{T}(\Gamma)$.
\item[$(v)$] Let $\Gamma$ be an ordered abelian group. Let $K:=\F_2((t^\Gamma))$ denote the Hahn series field and let $v$ be its canonical $t$-adic valuation. Then $K_{\mathcal{O}^\times_v}$ is isomorphic to $\mathcal{T}'(\Gamma)$.
\end{itemize}
\end{proposition}

\begin{proof}
Let $K$ be a field with more than two elements. Then $K_{K^\times}$ contains precisely two elements, the coset $0K^\times$ and the coset $1K^\times$, the latter containing all non-zero elements of $K$. If $x\in K^\times$, then $-x\in K^\times$ as well and thus $0K^\times$ belongs to $1K^\times+1K^\times$. Since by assumption there are $x,y\in K^\times$ with $x\neq y$, also $(x-y)K^\times\neq 0K^\times$ is an element of $1K^\times+1K^\times$. This shows $(i)$.

Let $K$ be an ordered field with positive cone $P$. Here, this means that $P$ is a subset of $K^\times$ such that $P+P,P\cdot P\subseteq P$ and $K^\times$ is the disjoint union of $P$ and $-P$. It follows that $K_P$ contains precisely three\footnote{In the literature, sometimes positive cones are defined as \lq\lq non-negative cones\rq\rq, i.e., containing $0$. Notice that for our statement to hold it is necessary to exclude $0$ from positive cones.} elements: $0P$, $1P$ (containing all the elements of $P$) and $-1P$ (containing all the elements of $-P$). Since $P+P\subseteq P$, we have that $1P+1P$ only contains $1P$. By definition, if $x\in P$, then $-x\in -P$. Thus, $1P-1P$ contains $0P$. It does also contain $1P$ and $-1P$ since e.g.\ $1+1\in P$ (because $1\in P$) and $(1+1)-1=1\in P$ while $1-(1+1)=-1\in -P$. This shows~$(ii)$.

Let $K$ be $\F_p$ (the finite field with $p$ elements) for some prime number $p>3$ such that $p\equiv 3\mod 4$. The prime number $p$ is certainly odd and thus the cardinality of $(K^\times)^2$ is $\frac{p-1}{2}$. If $-1$ is a square in $K$, then $K^\times$ contains an element of order $4$ but this is excluded by the assumption $p\equiv 3\mod 4$. We have that $x\in K^\times$ is a square if and only if $-x$ is not a square. It follows that $K_{(K^\times)^2}$ contains precisely three elements: $0(K^\times)^2$, $1(K^\times)^2$ (containing all the non-zero squares) and $-1(K^\times)^2$ (containing all the non-squares). Moreover, $1(K^\times)^2-1(K^\times)^2$ contains $0(K^\times)^2$. This shows that $-1(K^\times)^2$ is the hyperinverse of $1(K^\times)^2$ in $K_{(K^\times)^2}$. Since $1+1\neq 0$ in $K$ and $1=(1+1)-1$, we have that either $1(K^\times)^2$ or $-1(K^\times)^2$ (depending on which among $1+1$ and $-(1+1)$ is a square) belongs to $1(K^\times)^2-1(K^\times)^2$. By the symmetry of $1(K^\times)^2-1(K^\times)^2$ under the application of $-$ (i.e., multiplication by $-1(K^\times)^2$), we obtain that both $1(K^\times)^2$ and $-1(K^\times)^2$ must be elements of $1(K^\times)^2-1(K^\times)^2$. In a finite field, any non-square is the sum of two (non-zero) squares. Thus, $-1(K^\times)^2$ is an element of $1(K^\times)^2+1(K^\times)^2$. In order to show that $1(K^\times)^2$ is an element of $1(K^\times)^2+1(K^\times)^2$ as well, we distinguish two cases. If $1+1$ is a square in $K$, then $1(K^\times)^2$ is in $1(K^\times)^2+1(K^\times)^2$ trivially. Otherwise, $1+1$ is not a square. In this case, either $1+1+1$ is a square, in which case $1=(1+1+1)-(1+1)$ shows that $1(K^\times)^2$ is an element of $1(K^\times)^2+1(K^\times)^2$, or $-(1+1+1)=-((1+1)+1))$ is a square, in which case $-(1+1)=-(1+1+1)+1$ implies that $1(K^\times)^2$ is an element of $1(K^\times)^2+1(K^\times)^2$. This proves $(iii)$.

Let $k$ be a field with more thatn $2$ elements and consider the Hahn series field $K=k((t^\Gamma))$ equipped with the $t$-adic valuation $v$. A non-zero element $x$ of $K^\times$ is represented as a formal series
\[
x=\sum_{\gamma\in\Gamma}a_\gamma t^\gamma
\]
with well-ordered support, i.e., $\{\gamma\in\Gamma\mid a_\gamma\neq 0\}$ is a well-ordered set with respect to the order of $\Gamma$. The $t$-adic value of $x$ under $v$ is defined to be the minimum of the support of $x$ (cf.\ \cite[Exercise 3.5.6]{PE05}).

The value group of $v$ is thus $\Gamma$ and it follows from general valuation theory that $K^\times/\mathcal{O}_v^\times$ is isomorphic to $\Gamma$ as an ordered abelian group (cf.\ \cite[Section 2.1]{PE05} or Lemma \ref{Valuegroupasquotient} below). For $x\in K^\times$, we have that $x\mathcal{O}_v^\times$ contains all the elements of $K$ with the same value of $x$ under $v$. If $x,y\in K$ are such that $vx\neq vy$ (i.e., $x\mathcal{O}_v^\times\neq y\mathcal{O}_v^\times$), then
\[
v(x+yt)=v(x+y)=\min\{vx,vy\}
\]
for any $t\in\mathcal{O}_v^\times$ (i.e., any $t$ with $vt=0$). This follows from \cite[Section 1.3 (1.3.4)]{PE05} (or Corollary \ref{valringFVK} $(iv)$ below). We now note that if $x\in K$ has any positive value under $v$, then $v(x-1)=0$ and thus $vx=v(1+(x-1))$ is the value of some element of $K$ generating a coset in $K_{\mathcal{O}_v^\times}$ which belongs to $1\mathcal{O}_v^\times+1\mathcal{O}_v^\times$. Moreover, by the assumption on the cardinality of $k$, there exists $a\in k\setminus\{1\}$, so by the definition of the $t$-adic valuation, we have that $va=0$ and $v(1-a)=0$. Since $v(1)=v(-1)=0$, we conclude that the image of $1\mathcal{O}_v^\times-a\mathcal{O}_v^\times=1\mathcal{O}_v^\times+1\mathcal{O}_v^\times$ under $v$ is $[0,\infty]$. By distributivity, this suffices to show that $\mathcal{T}(\Gamma)\simeq K_{\mathcal{O}_v^\times}$ (see also Lemma \ref{vT} below). The proof of $(iv)$ is now also complete.

The proof for $(iv)$ can easily be adapted to prove $(v)$.
\end{proof}

\begin{remark}\label{nonquot}
Let us mention at this point that not all hyperfields are factor hyperfields. This fact was proved by Massouros in \cite{Mas85} who then further improved his results in \cite{Mas1985}. According to \cite[Theorem 6]{LKS23} finite hyperfields such that $1\notin 1+1$ and with the property that $0$ does not belong to any finite sum of $1$ with itself are also not quotient hyperfields.
\end{remark}

\subsection{Subhyperrings}

To choose a good notion of subhyperring is not as straightforward as it might seem. Two options are presented in the next definition.

\begin{definition}\label{subhr}
Let $(R,+,\cdot,0)$ be a hyperring. A subset $S$ of $R$ is a \emph{relational subhyperring} of $R$ if it is multiplicatively closed and with the \emph{induced} multivalued operation:
\begin{align*}
    x+_S y&:=(x+y)\cap S\quad(x,y\in S)
\end{align*}
we have that $(S,+_S,\cdot,0)$ is a hyperring as well.\par
A subset $S$ of $R$ is a \emph{(traditional) subhyperring} of $R$ if $0\in S$ and for all $x,y\in S$ we have that $x-y\subseteq S$ and $xy\in S$.\par
If $R$ is a hyperring with unity $1$, then we say that a relational subhyperring $S$ of $R$ is a \emph{relational subhyperfield} of $R$ if $1\in S$ and $(S,+_S,\cdot,0,1)$ is a hyperfield. In addition, a subhyperring $S$ of $R$ is called a \emph{(traditional) subhyperfield} if $1\in S$ and $(S\setminus\{0\},\cdot,1)$ is an abelian group.
\end{definition}

\begin{remark}
The adjective \lq\lq relational\rq\rq\ has been chosen for the following reason. A multivalued operation can be encoded in a first-order language via the ternary relation $z\in x+y$ (as noticed e.g.\ in \cite{Lee20}). If we consider a first-order language $\mathcal{L}$ with a constant symbol for $0$, a binary function symbol for $\cdot$ and a ternary relation symbol for $+$, then a hyperring naturally becomes a structure on $\mathcal{L}$, which is a model of all the axioms of Definitions \ref{Krasnerhyperrings} and \ref{hypergp} (written as sentences in $\mathcal{L}$). Under this interpretation, the relational subhyperrings of a hyperring $R$ are precisely the submodels of $R$, i.e., the substructures of $R$ (see \cite[Section 2.3]{PD11}) which satisfy the axioms. For more details on the model theoretical point of view let us refer the reader to \cite{PhD22}.\par
Our choices regarding terminology are also motivated by historical reasons. In fact, subhypergroups (and consequently subhyperfields) have been defined long time ago (see e.g.\ Definition 2 and the subsequent remark in \cite{Kra37}\footnote{This article is included in Krasner's collected works \cite[pages 280--406]{BP19}}).   
\end{remark}

\begin{remark}
It is clear that a subhyperring $S$ of a hyperring $(R,+,\cdot,0)$ is a relational subhyperring of $R$ and that $+_S=+$ in this case. On the other hand, not all relational subhyperrings are subhyperrings, as the following example shows.
\end{remark}

\begin{example}\label{subvsstrictsub}
Consider an ordered abelian group $(\Gamma,+,0,<)$. Let $R:=\mathcal{T}(\Gamma)$ be the hyperfield obtained as in Example \ref{TGamma} and consider the subset $S:=\{\infty,0\}$ of $R$. It is straightforward to check that $S$ is a relational subhyperfield of $R$. However, $S$ is not a subhyperfield of $R$ since $0\boxplus 0=[0,\infty]$ is not a subset of $S$.
\end{example}

\begin{remark}\label{subhypfTG}
Let $\Gamma$ be an ordered abelian group. It is not difficult to see that the only subhyperfield of $\mathcal{T}(\Gamma)$ is $\mathcal{T}(\Gamma)$ itself. On the other hand, if $\Delta$ is a subgroup of $\Gamma$, then $\mathcal{T}(\Delta)$ is a relational subhyperfield of $\mathcal{T}(\Gamma)$ (Example \ref{subvsstrictsub} above corresponds to the case $\Delta=\{0\}$). 
\end{remark}

\subsection{Homomorphisms}

Next, let us consider a notion of homomorphism for hyperrings and hyperfields.

\begin{definition}\label{homomorphism}
Let $(R,+^R,\cdot^R,0_R)$ and $(S,+^S,\cdot^S,0_S)$ be hyperrings. We say that a map $\sigma:R\to S$ is a \emph{homomorphism of hyperrings} if
\begin{itemize}
\item[(HH1)] $\sigma(0_R)=0_S$,
\item[(HH2)] $\sigma(x\cdot^R y)=\sigma(x)\cdot^S \sigma(y)$, for all $x,y\in R$,
\item[(HH3)] $\sigma(x+^R y)\subseteq \sigma(x)+^S \sigma(y)$, for all $x,y\in R$.
\end{itemize} 
If $(R,+^R,\cdot^R,0_R,1_R)$ and $(S,+^S,\cdot^S,0_S,1_S)$ are hyperfields, then we say that a homomorphism of hyperrings $\sigma:R\to S$ is a \emph{homomorphism of hyperfields} when the following properties
\begin{itemize}
\item[(HH4)] $\sigma(1_R)=1_S$,
\item[(HH5)] $\sigma(x^{-1})=\sigma(x)^{-1}$, for all $x\in R\setminus\{0\}$,
\end{itemize} 
hold as well.
\end{definition}

\begin{example}
Let $R$ be a hyperring. The map defined by $\sigma(0):=0$ and $\sigma(x):=1$ for $x\in R\setminus\{0\}$ is a homomorphism of hyperrings $R\to\K$. 
\end{example}

\begin{example}
Let $(\Gamma,<,+,0_\Gamma)$ be an ordered abelian group. The map 
\begin{align*}
\iota:\K&\to\mathcal{T}(\Gamma)\\
0_\K&\mapsto\infty\\
1_\K&\mapsto0_\Gamma
\end{align*}
is a homomorphism of hyperfields.
\end{example}

\begin{example}
Let $K$ be a field and $T$ a subgroup of $K^\times$. The function $K\to K_T$ mapping $x\in K$ to $[x]_T\in K_T$ is a homomorphism of hyperfields.
\end{example}

\begin{remark}\label{ker}
Let $\sigma:R\to S$ be a homomorphism of hyperrings. Then the \emph{kernel} of $\sigma$
\[
\ker\sigma:=\{x\in R\mid \sigma(x)=0_S\}
\]
is a subhyperring of $R$. Indeed, for $x,y\in\ker\sigma$, we have that
\[
\sigma(x-y)\subseteq \sigma(x)-\sigma(y)=0_S-0_S=\{0_S\}.
\]
Thus, $\sigma(z)=0_S$ for all $z\in x-y$, i.e., $x-y\subseteq\ker\sigma$.
\end{remark}

\begin{remark}
Let $(R,+,\cdot,0,1)$ be a hyperring and $S\subseteq R$. Consider the inclusion map $\iota:S\hookrightarrow R$. It is straightforward to verify that if $S$ is a relational subhyperring of $R$, then $\iota$ is a homomoprhism of hyperrings.\par
On the other hand, if $(R,+^R,\cdot^R,0_R)$ and $(S,+^S,\cdot^S,0_S)$ are hyperrings with $S\subseteq R$ and the inclusion $\iota:S\hookrightarrow R$ is a homomorphism of hyperrings, then one cannot conclude that $S$ is a relational subhyperring of $R$, as the following example shows.
\end{remark}

\begin{example}\label{SW}
The identity $\mathbb{S}\to\mathbb{W}$ is a homomorphism of hyperrings. However, $\mathbb{S}$ is not a subhyperring of $\mathbb{W}$ as the hypersum of $1$ with itself in $\mathbb{S}$ is $\{1\}$, while the hypersum of $1$ with itself in $\mathbb{W}$ intersected with $\mathbb{S}$~is~$\{-1,1\}$.
\end{example}

The above example motivates the following definition.

\begin{definition}
Let $(R,+^R,\cdot^R,0_R)$ and $(S,+^S,\cdot^S,0_S)$ be hyperrings (resp.\ hyperfields with unities $1_R$ and $1_S$). An injective homomormphism of hyperrings (resp.\ hyperfields) $\sigma:R\to S$ is called an \emph{embedding of hyperrings (resp. hyperfields)} if
\begin{itemize}
\item[(EM1)] $\sigma(x+^R y)=(\sigma(x)+^S \sigma(y))\cap\Ima\sigma$, for all $x,y\in R$.
\end{itemize}
\end{definition}

We leave the straightforward proof of following lemma to the reader.

\begin{lemma}
If $S$ is a relational subhyperring of a hyperring $R$, then the inclusion map $S\hookrightarrow R$ is an embedding of hyperrings.
Conversely, If $R$ and $S$ are hyperrings with $S\subseteq R$ and the inclusion map $S\hookrightarrow R$ is an embedding of hyperrings, then $S$ is a relational subhyperring of $R$.
\end{lemma}

\begin{definition}\label{isomorphism}
A homomorphism of hyperrings (resp.\ hyperfields) $\sigma:R\to S$ is called an \emph{isomorphism of hyperrings (resp.\ hyperfields)} if it is a bijective map and $\sigma^{-1}:S\to R$ is also a homomorphism of hyperrings (resp.\ hyperfields). We say that two hyperrings (resp.\ hyperfields) $R$ and $S$ are \emph{isomorphic}  and we write $R\simeq S$ if there exists an isomorphism $\sigma:R\to S$.
\end{definition}

\begin{lemma}\label{isosuremb}
Let $R$ and $S$ be hyperrings (resp.\ hyperfields) and denote by $+^R$ and $+^S$ their hyperoperations, respectively. A map $\sigma:R\to S$ is an isomorphism of hyperrings (resp.\ hyperfields) if and only if $\sigma$ is a surjective embedding of hyperrings (resp.\ hyperfields).
\end{lemma}

\begin{proof}
If $\sigma$ is an isomorphism, then it is bijective by definition, in particular $\Ima\sigma=S$. Since $\sigma$ is a homomorphism, we have that $\sigma(x+^R y)\subseteq\sigma(x)+^S\sigma(y)$ holds. On the other hand, since $\sigma^{-1}$ satisfies (HH3) as well, we obtain that
\[
\sigma^{-1}(\sigma(x)+^S\sigma(y))\subseteq  x+^Ry.
\]
Applying $\sigma$ to both sides then yields $\sigma(x+^R y)\supseteq\sigma(x)+^S\sigma(y)$.\par
Assume now that $\sigma$ is a surjective embedding and let us show that $\sigma^{-1}$ is a homomorphism. Clearly, (HH1) and (HH2) hold for $\sigma^{-1}$. Since by (EM1) and the surjectivity we have that $\sigma(x+^Ry)=\sigma(x)+^S\sigma(y)$, property (HH3) follows applying $\sigma^{-1}$ to both sides and using again the surjectivity of $\sigma$.
\end{proof}

\begin{remark}
Note that a bijective homomorphism is not necessarily an isomorphism as it can be deduced from Example~\ref{SW}.
\end{remark}

\begin{example}\label{iso}
Let $(\Gamma,<,+,0_\Gamma)$ be an ordered abelian group. Consider the relational subhyperfield $S:=\{\infty,0\}$ of $\mathcal{T}(\Gamma)$ as in Example \ref{subvsstrictsub}. Then $(S,+_S,\cdot,\infty,0_\Gamma)$ is isomorphic to $\K$.
\end{example}

As usual we identify isomorphic structures: Examples \ref{subvsstrictsub} and \ref{iso} can be expressed by saying that $\K$ is a subhyperfield of $\mathcal{T}(\Gamma)$ for any ordered abelian group $\Gamma$.

\subsection{Hyperideals}
We now briefly discuss how the classical ring theory notion of ideal generalises in the multivalued setting.

\begin{definition}[Definition 2.1 in \cite{BC17} and Definition 2.9 in \cite{KLS22}]
Let $R$ be a hyperring. A subhyperring $I$ of $R$ is a
\emph{hyperideal of $R$} 
if it satisfies
\begin{itemize}
\item[(HID1)]  $xy \in I$, for all $x \in R$ and $y \in I$. 
\end{itemize}
\end{definition}

\begin{lemma}\label{hypid}
Let $R$ be a hyperring with no zero divisors, i.e., $xy=0_R$ implies $x=0_R$ or $y=0_R$. A relational subhyperring $I$ of $R$ satisfying (HID1) is a hyperideal of $R$.
\end{lemma}

\begin{proof}
Pick $x,y\in I$ and let $z\in x-y$. It suffices to prove that $z\in I$. If $x=0_R$ or $y=0_R$, then there is nothing to show. Otherwise, by distributivity in $R$ we have that $zy\in xy-y^2$ and by (HID1) we obtain that $zy\in I$, since $y\in I$. Therefore, $zy\in (xy-y^2)\cap I$. Now, distributivity in $I$ (with respect to the induced multivalued operation $+_I$) yields $zy\in\bigl((x-y)\cap I\bigr)y$, so there exists $z'\in(x-y)\cap I$ such that $zy=z'y$. Hence, $0_R\in zy-z'y=(z-z')y$. Since $R$ has no zero divisors, this implies that $0_R\in z-z'$ and hence $z=z'\in I$. 
\end{proof}

\begin{remark}\label{Jun-ker}
We wish to justify the choice of requiring hyperideals to be traditional subhyperrings. In \cite[Section 3.1]{Jun18}, Jun provides a generalisation of the classical quotient construction of a ring modulo an ideal in the multivalued setting. One property that is certainly desirable to preserve is for the canonical epimorphism from a hyperring $R$ to the quotient hyperring $R/I$ modulo a hyperideal $I$ to be a homomorphism of hyperrings $R\to R/I$ having $I$ as its kernel. We have no examples of hyperrings with a relational subhyperring satisfying (HID1), but at the same time we did not succeed in proving Lemma \ref{hypid} without the assumption on zero divisors. If such an example would exist, then that relational subhyperring could not be the kernel of a homomorphism of hyperrings by Remark \ref{ker} and so we should not call it a hyperideal.
\end{remark}

\begin{remark}
Let $\sigma:R\to S$ be a homomorphism of hyperrings. Then it is easy to show that $\ker\sigma$ is a hyperideal of $R$. Conversely, for all hyperideals $I$ of a hyperring $R$ the map $\sigma:R\to\K$ defined as
\[
\sigma(x):=\begin{cases}0&\text{if }x\in I,\\ 1&\text{ otherwise.}\end{cases}
\]
is a homomorphism of hyperrings such that $\ker\sigma=I$.
\end{remark}

The following statements can be shown to hold as in the classical theory of rings.
\begin{lemma}[Lemma 2.12 in \cite{KLS22}]\label{easy}
If a hyperideal $I$ of a hyperring $R$ contains a unit, then $I=R$.
\end{lemma}

\begin{corollary}[Corollary 2.13 in \cite{KLS22}]\label{hypidfields}
The only hyperideals of a hyperfield $F$ are $\{0\}$ and $F$.
\end{corollary}

\begin{definition}[Definition 2.14 \textit{(ii)} in \cite{KLS22}]
Let $I$ be a hyperideal of a hyperring $R$. Then $I$ is called \emph{maximal} if $I\subsetneq R$ and for all hyperideals $J$ of $R$ we have that $I\subsetneq J$ implies $J=R$.
\end{definition}

We recall without proof the following result.

\begin{proposition}[Proposition 2.16 $(ii)$ in \cite{KLS22}]\label{maxid}
Assume that $R$ is a hyperring with unity. Let $I$ be a hyperideal $R$. Then $I$ is maximal if and only if $R/I$ is a hyperfield.
\end{proposition}

\begin{example}[\cite{Mit73}]\label{scalars}
Let $R$ be a hyperring with unity. An element $s$ of $R$ is called a \emph{scalar} if $x+s$ is a singleton for all $x\in R$. As in the proof of Lemma \ref{1-1=0}, it is not difficult to show that $s\in R$ is a scalar if and only if $s-s=\{0\}$. Interestingly, the set of all scalar elements of $R$ forms a hyperideal of $R$.\par 
Assume now that $R$ is a hyperring with unity. If $1$ is a scalar, then by Lemma \ref{easy} the hyperideal of scalars is $R$ and thus $R$ is a ring (as all of its elements are scalars). Moreover, we deduce from Corollary \ref{hypidfields} that if $F$ is a hyperfield, but not a field, then the only scalar of $F$ is $0$, i.e., $x-x$ is not a singleton for all $x\in F^\times$.
\end{example}

\section{Valued hyperfields}\label{sec2}

Some of the results presented in this section can be found in \cite{KLS22}. Some proofs have been improved, others we will repeat for the convenience of the reader.

\subsection{Valuations}
The next definition is a straightforward generalisation of the definition of valuation for fields.
\begin{definition}[Definition 4.1 in \cite{KLS22}]\label{hyperval}
Take a hyperfield $F$ and an ordered abelian group $\Gamma$ (written additively). 
A surjective map $v:  F\to \Gamma \cup \{\infty\}$ is called a \emph{valuation on $F$} if it has the following properties:
\begin{itemize}
\item[(V1)] $vx = \infty\iff x = 0$, for all $x\in F$,
\item[(V2)] $v(xy) = vx+ vy$, for all $x,y\in F$,
\item[(V3)] $z \in x+y~\Longrightarrow ~vz \ge \min\{vx,vy\}$, for all $x,y,z\in F$.
\end{itemize}
If $v$ is a valuation on a hyperfield $F$ we call $(F,v)$ a \emph{valued hyperfield} and we denote by $vF$ the ordered abelian group $v(F^\times)$, i.e., the \emph{value group} of $(F,v)$.
\end{definition}

\begin{example}\label{extriv}
Let $F$ be a hyperfield. The \emph{trivial} valuation $vx=0$ for all $x\in F^\times$ is always a valuation on $F$ with value group $\{0\}$.
\end{example}

The next result provides further examples of valued hyperfields.

\begin{lemma}[Corollary 4.2 in \cite{KLS22}]\label{vT}
Let $(K,v)$ be a valued field and take a subgroup $T$ of $K^\times$. Denote by $\mathcal{O}_v^\times$ the group of units of the valuation ring of $K$, i.e., $\mathcal{O}_v^\times:=\{x\in K\mid vx=0\}$.
If $T\subseteq\mathcal{O}_v^\times$, then
\begin{align*}
v_T: K_T&\to vK\cup\{\infty\}\\
xT&\mapsto vx 
\end{align*}
is a valuation on $K_T$.
\end{lemma}

\begin{proof}
The assumption $T\subseteq\mathcal{O}_v^\times$ guarantees that the map $v_T$ is well-defined. Moreover, surjectivity, (V1) and (V2) follow immediately from the corresponding properties of $v$ and the definition of the factor hyperfield $K_T$. It remains to verify (V3). Take $x,y\in K$ and $t\in T$. We obtain that
\[
v(x+yt)\geq\min\{vx,v(yt)\}=\min\{vx,vy+vt\}=\min\{vx,vy\}.
\]
From this it easily follows that (V3) holds for $(K_T,v_T)$.
\end{proof}

The next lemma gives an alternative definition of valuation on (hyper)fields (cf.\ \cite[Example 1.8]{BL21}).

\begin{lemma}\label{valhom}
Let $F$ be a hyperfield and $\Gamma$ an ordered abelian group. Then $(F,v)$ is a valued hyperfield with $vF=\Gamma$ if and only if the map $v:F\to\mathcal{T}(\Gamma)$ is a surjective homomorphism of hyperfields.
\end{lemma}
\begin{proof}
Assume that $(F,v)$ is a valued hyperfield. Then $v:F\to\mathcal{T}(vF)$ is a surjective map, (HH1) follows from (V1), (HH2) follows from (V2), (HH3) follows from (V3). Property (HH4) follows because $vx=vx+v(1)$ for all $x\in F$ and (HH5) follows since, for $x\in F^\times$, we have that 
\[
0=v(1)=v(xx^{-1})=vx+v(x^{-1}).
\] 
Thus, $v$ is a homomorphism of hyperfields.\par
Conversely, let $F$ be a hyperfield and $v:F\to\mathcal{T}(\Gamma)$ a surjective homomoprhism of hyperfields. Property (V2) for $v$ follows from (HH2) and (V3) follows from (HH3) together with the definition of the hyperoperation of $\mathcal{T}(\Gamma)$. Property (V1) follows from Corollary \ref{hypidfields} since $v(1)=0\neq\infty$ and thus $\ker v=\{x\in F\mid vx=\infty\}$ is a proper hyperideal of $F$. We proved that $v$ is a valuation on $F$ with $vF=\Gamma$.
\end{proof}

On the basis of the previous lemma, the following result is almost immediate to verify. Thus, we~omit~its~proof.

\begin{corollary}[Lemma 4.5 in \cite{KLS22}]\label{valringFVK}
Let $v:F\to\Gamma\cup\{\infty\}$ be a valuation on a hyperfield $F$. Then:
\begin{enumerate}
 \item[$(i)$] $v(1)=v(-1) = 0$,
 \item[$(ii)$] $v(-x)=vx$ for all $x\in F$,
 \item[$(iii)$] $vx^{-1}=-vx$ for all $x\in F^\times$,
 \item[$(iv)$] if $vx\neq vy$, then 
 $vz = \min\{vx, vy\}$, for every $x,y\in F$ and $z \in x+y$.
\end{enumerate}
\end{corollary}

\begin{lemma}\label{KTvalK}
Let $(K_T,w)$ be a valued hyperfield which is a factor hyperfield. Then there exists a valuation $v$ on $K$ such that $T\subseteq\mathcal{O}_v^\times$ and $w=v_T$.
\end{lemma}

\begin{proof}
Let $\Gamma$ be the value group of $(K_T,w)$. Since $K\to K_T$, $x\mapsto[x]_T$ and $w:K_T\to\mathcal{T}(\Gamma)$ are surjective homomorphisms of hyperfields, their composition $v:K\to\mathcal{T}(\Gamma)$ is a valuation on $K$ satisfying the conditions of the statement. 
\end{proof}

\begin{example}
By Proposition \ref{exquot} $(iv)$ we have that $\mathcal{T}(\Gamma)$ is isomorphic to all factor hyperfields of the form $k((t^\Gamma))_{\mathcal{O}_{v_t}^\times}$ for some field $k$. By the above lemmas, the isomorphism $\sigma:k((t^\Gamma))_{\mathcal{O}_{v_t}^\times}\to\mathcal{T}(\Gamma)$ is a valuation on $k((t^\Gamma))_{\mathcal{O}_{v_t}^\times}$ and the identity map $\Id=\sigma\circ\sigma^{-1}:\mathcal{T}(\Gamma)\to\mathcal{T}(\Gamma)$ is a valuation on $\mathcal{T}(\Gamma)$.
\end{example}

\subsection{Valuation hyperrings}

The next definition is inspired by classical valuation theory.
\begin{definition}[Definition 4.6 in \cite{KLS22}]
Let $F$ be a hyperfield. A relational subhyperring $\mathcal{O}$ of $F$ is called a \emph{valuation hyperring} in $F$ if for all $x\in F^\times$ we have that either $x\in \mathcal{O}$ or $x^{-1}\in\mathcal{O}$.
\end{definition}
Observe that, by definition, it follows that $1\in\mathcal{O}$ for any valuation hyperring $\mathcal{O}$ in $F$. Let us now prove some more basic properties of valuation hyperrings.
\begin{lemma}[Proposition 4.7 in \cite{KLS22}]\label{valstrict}
A valuation hyperring $\mathcal{O}$ in a hyperfield $F$ is a subhyperring of $F$.
\end{lemma}
\begin{proof}
It suffices to show that $a-b\subseteq \mathcal{O}$ for all $a,b\in\mathcal{O}$. Take $a,b\in\mathcal{O}$ and $x\in a-b$. If $x\in \mathcal{O}$, then there is nothing to show (note that this case also includes $x=0$). Otherwise, we have that $x^{-1}\in\mathcal{O}$ and thus $ax^{-1},bx^{-1}\in\mathcal{O}$. Since $x\in a-b$ we obtain from (CH4) that $a\in x+b$, so, using axiom (HR3),
\[
ax^{-1}\in(x+b)x^{-1}=1+bx^{-1}.
\]
We have obtained that $ax^{-1}\in (1+bx^{-1})\cap\mathcal{O}=1+_{\mathcal{O}}bx^{-1}$. By (CH4) and (HR3) applied to the hyperring $(\mathcal{O},+_\mathcal{O},\cdot,0)$, it follows that 
\[
xx^{-1}=1\in ax^{-1}+_\mathcal{O}(-bx^{-1})=(a+_\mathcal{O}(-b))x^{-1}.
\] 
Therefore, $x\in a+_\mathcal{O}(-b)\subseteq\mathcal{O}$. This shows that $a-b\subseteq\mathcal{O}$. 
\end{proof}

\begin{lemma}[Lemma 4.8 in \cite{KLS22}]\label{quotient}
Let $\mathcal{O}$ be a valuation hyperring in a hyperfield $F$. Then \mbox{$\mathcal{M}:=\mathcal{O}\setminus\mathcal{O}^\times$} is the unique maximal hyperideal of $\mathcal{O}$.
\end{lemma}
\begin{proof}
Take $a\in\mathcal{M}$ and $c\in \mathcal{O}$. If $ca$ is invertible in $\mathcal{O}$, then there exists $x\in\mathcal{O}$ such that $x(ca)=1$. Hence $(xc)a=1$ and $a^{-1}=xc\in\mathcal{O}$ contradicting $a\in\mathcal{M}$. This proves that $ca \in \mathcal{M}$ and shows that $\mathcal{M}$ satisfies (HID1).\par
Take $a,b\in\mathcal{M}$. We may assume that $ab^{-1}\in\mathcal{O}$ (otherwise $ba^{-1}\in\mathcal{O}$ and we can interchange the roles of $a$ and $b$). Since $\mathcal{O}$ is a subhyperring of $F$ (cf.\ Lemma \ref{valstrict}), we obtain that $1-ab^{-1}\subseteq \mathcal{O}$ and therefore, using what we have just proved, we conclude that 
\[
b-a =b(1-ab^{-1})\subseteq \mathcal{M}.
\] 
We have shown that $\mathcal{M}$ is a hyperideal of $\mathcal{O}$.\par Since, by the definition of $\mathcal{M}$, we have that $\mathcal{O}\setminus\mathcal{M}=\mathcal{O}^\times$, by Lemma \ref{easy}, every proper hyperideal of $\mathcal{O}$ must be contained in $\mathcal{M}$, showing that $\mathcal{M}$ is the unique maximal hyperideal of $\mathcal{O}$.
\end{proof}

\subsection{Residue hyperfield}

Any valuation on a hyperfield $F$ induces a valuation hyperring in $F$.

\begin{proposition}[Proposition 4.11 in \cite{KLS22}]\label{OvMv}
Let $v:F \to \Gamma \cup \lbrace \infty \rbrace$ be a valuation on a hyperfield $F$. Then 
\[
\mathcal O_v: = \lbrace x \in F \mid vx \geq 0 \rbrace
\]
is a valuation hyperring in $F$ and 
\[
\mathcal M_v := \lbrace x \in F \mid vx > 0 \rbrace
\]
is its unique maximal hyperideal.
\end{proposition}
\begin{proof}
We first prove that $\mathcal{O}_v$ is a subhyperring of $F$. Take $a,b\in \mathcal{O}_v$. By (V3), for all $c\in a-b$ we have $vc\geq\min\{va,v(-b)\}=\min\{va,vb\}\geq 0$, so $a-b\subseteq \mathcal{O}_v$. Further, we have $ab\in\mathcal{O}_v$ by (V2). By Corollary \ref{valringFVK} $(iii)$ we conclude that if $x\notin\mathcal{O}_v$, then $x^{-1}\in\mathcal{O}_v$ so $\mathcal{O}_v$ is a valuation hyperring in $F$.\par
Next we show that $\mathcal{M}_v$ is the unique maximal hyperideal of $\mathcal{O}_v$. Observe that, by virtue of Corollary \ref{valringFVK} $(iii)$, 
\[
\mathcal{O}_v^\times=\{x\in\mathcal{O}_v\mid vx=0\}.
\]
Hence, $\mathcal{M}_v=\mathcal{O}_v\setminus\mathcal{O}_v^\times$ and then $\mathcal{M}_v$ is the unique maximal hyperideal of $\mathcal{O}_v$ by Lemma \ref{quotient}.
\end{proof}

\begin{remark}\label{descrOv}
Let $(F,v)$ be a valued hyperfield. Note that $\mathcal{O}_v$ can be described in terms of the hyperoperation $\boxplus$ of $\mathcal{T}(vF)$ as the preimage in $F$ of $v(1)\boxplus v(1)$ under $v$:
\[
\mathcal{O}_v=v^{-1}(v(1)\boxplus v(1)).
\]
\end{remark}

It follows from Proposition \ref{maxid} that, for a valuation hyperring $\mathcal{O}$ with maximal hyperideal $\mathcal{M}$, the quotient hyperring 
\[
\mathcal{O}/\mathcal{M}=\{x+\mathcal{M}\mid x\in\mathcal{O}\},
\] 
with the multivalued operation $\oplus$ defined as
\[
(x+\mathcal{M})\oplus(y+\mathcal{M}):=\{z+\mathcal{M}\mid z\in x+y\},
\]
is a hyperfield (see \cite[Section 3]{Jun18} and Remark \ref{Jun-ker} above).

\begin{definition}
If $(F,v)$ is a valued hyperfield, then we call the hyperfield $\mathcal{O}_v/\mathcal{M}_v$ the \emph{residue hyperfield} of $(F,v)$ and we denote it by $Fv$. For an element $x\in\mathcal{O}_v$, we denote by $xv$ its natural image in $Fv$.
\end{definition}

\begin{proposition}\label{resquot}
Let $(K,v)$ be a valued field  and $T\subseteq\mathcal{O}_v^\times$ a subgroup of $K^\times$. Then $K_Tv_T\simeq (Kv)_{(Tv)}$, where $Tv:=\{tv\mid t\in T\}$.
\end{proposition}

\begin{proof}
By definition $v_T[x]_T=vx$ for all $x\in K$, thus $\mathcal{O}_{v_T}=(\mathcal{O}_{v})_T$ and $\mathcal{M}_{v_T}=(\mathcal{M}_{v})_T$. It follows that
\[
K_Tv_T=\mathcal{O}_{v_T}/\mathcal{M}_{v_T}=(\mathcal{O}_{v})_T/(\mathcal{M}_{v})_T.
\]
Let us define
\begin{align*}
\sigma: (\mathcal{O}_{v})_T/(\mathcal{M}_{v})_T&\to (Kv)_{(Tv)}\\
[x]_T+(\mathcal{M}_v)_T&\mapsto [xv]_{Tv}
\end{align*}
and show that $\sigma$ is an isomorphism of hyperfields. By \cite[Lemma 3.3]{Jun18} we have that $[x]_T+(\mathcal{M}_v)_T=[y]_T+(\mathcal{M}_v)_T$ if and only if $v(x-yt)>0$ for some $t\in T$ . This implies that 
\[
0=(x-yt)v=xv-yv\cdot tv\in[xv]_{Tv}-[yv]_{Tv}~,
\] 
where we have used the assumption $T\subseteq\mathcal{O}_v^\times$. This shows that $\sigma$ is well-defined and injective. The surjectivity of $\sigma$ is clear. Moreover, $\sigma$ is easily seen to be a homomorphism of the corresponding multiplicative groups. Finally, we observe that
\begin{align*}
\sigma\bigl([x]_T+(\mathcal{M}_v)_T\oplus([y]_T+(\mathcal{M}_v)_T)\bigr)&=\{\sigma\bigl([z]_T+(\mathcal{M}_v)_T\bigr)\mid [z]_T\in [x]_T+[y]_T\}\\
&=\{[zv]_{Tv}\mid z=x+yt\text{ for some }t\in T\}\\
&=\{[xv+yv\cdot tv]_{Tv}\mid t\in T\}\\
&=[xv]_{Tv}+[yv]_{Tv}
\end{align*}
hence $\sigma$ is an isomorphism of hyperfields by Lemma \ref{isosuremb}.
\end{proof}

\begin{example}
By Proposition \ref{exquot}, $\mathcal{T}(\Gamma)\simeq K_{\mathcal{O}_v^\times}$, where $K=k((t^\Gamma))$ for some field $k$ with more than two elements and $v$ is its canonical $t$-adic valuation. Since $Kv=k$ and $\mathcal{O}_v^\times v=k^\times$, by the above proposition and Proposition \ref{exquot} $(i)$ we have that the residue field of $\mathcal{T}(\Gamma)$ with respect to its valuation given by the identity map is isomorphic to $k_{k^\times}\simeq\K$ which, moreover, is a relational subhyperfield of $\mathcal{T}(\Gamma)$ (cf.\ Examples \ref{subvsstrictsub} and~\ref{iso}).
\end{example}

For a valued field $(K,v)$ we denote by $1+\mathcal{M}_v$ the set of $1$\emph{-units}. That is, those $x\in K$ such that $v(x-1)>0$ or equivalently, $xv=1v$.

\begin{proposition}
Let $(K,v)$ be a valued field and $T\subseteq\mathcal{O}_v^\times$ a subgroup of $K^\times$. Then the map
\begin{align*}
\iota:K_Tv_T&\to K_T\\
[xv]_{Tv}&\mapsto [x]_T
\end{align*}
is an embedding of hyperfields if and only if $1+\mathcal{M}_v\subseteq T$.
\end{proposition}

\begin{proof}
If $1+\mathcal{M}_v\subseteq T$ and $x,y\in\mathcal{O}_v^\times$, then $[xv]_{Tv}=[yv]_{Tv}$ if and only if $xv=yv\cdot tv=(yt)v$ for some $t\in T$ if and only if $v(1-ytx^{-1})=v(x-yt)>0$ if and only if $ytx^{-1}\in 1+\mathcal{M}_v\subseteq T$. It follows that $[y]_T[x]_T^{-1}=[ytx^{-1}]_T=[1]_T$ and $\iota$ is well-defined. On the other hand if $[x]_T=[y]_T$ for some $x,y\in\mathcal{O}_v^\times$, then for some $t\in T$ we have that $x=yt$ and thus $xv=(yt)v=yv\cdot tv$. It follows that $\iota$ is injective. Moreover, we have that
\[
\iota\left([xv]_{Tv}\cdot[yv]^{-1}_{Tv}\right)=\iota\left([xv\cdot y^{-1}v]_{Tv}\right)=[xy^{-1}]_T=[x]_T[y]^{-1}_T
\]
and hence $\iota$ satisfies (HH2) and (HH5). It clearly satisfies (HH1) and (HH4). It remains to show that (EM1) holds. First observe that $\Ima\iota$ consists of classes $[x]_T$ where $x=0$ or $x\in R\subseteq\mathcal{O}_v^\times$, where $R$ is a set of some selected representatives for $Kv$.
Take again $x,y\in\mathcal{O}_v^\times$. We have that
\[
\iota\left([xv]_{Tv}+[yv]_{Tv}\right)=\{[x+yt]_T\mid x,y\in R,t\in T\cap R\}=\left([x]_T+[y]_T\right)\cap\Ima\iota.
\]
This completes the proof of one implication.\par
If $a\in 1+\mathcal{M}_v$ is not in $T$, then $av=1v$ and thus $[av]_{Tv}=[1v]_{Tv}$ holds. On the other hand, we have that
\[
\iota[av]=[a]_T\neq[1]_T=\iota[1v]_T
\] 
and thus $\iota$ is not well-defined and in particular not an embedding of hyperfields.  
\end{proof}

\subsection{Equivalence of valuations}

In analogy with classical valuation theory, our aim is now to show that valuation hyperrings can be used to describe valuations up to composition with an order preserving isomorphism of the value group.

\begin{proposition}[Proposition 4.12 in \cite{KLS22}]\label{FtimesOtimes}
Let $F$ be a hyperfield and $\mathcal{O}$ a valuation hyperring in $F$. Consider the multiplicative group $\Gamma:=F^\times/\mathcal{O}^\times$ and define a relation $\leq$ on $\Gamma$ as follows:
\[
a\mathcal{O}^\times\leq b\mathcal{O}^\times\iff ba^{-1}\in\mathcal{O}.
\]
Then $(\Gamma,\cdot,\leq)$ is an ordered abelian group and the canonical projection 
\[
\pi:F\to\Gamma\cup\{\infty\},
\] 
extended so that $\pi(0_F)=\infty$, is a valuation on $F$. Furthermore, $\mathcal{O}_\pi=\mathcal{O}$.
\end{proposition}
\begin{proof}
First we show that $\leq$ is an ordering for $(\Gamma,\cdot)$. Since $aa^{-1}=1_F\in\mathcal{O}$, reflexivity is clear. If $ab^{-1},ba^{-1}\in\mathcal{O}$, then $ab^{-1}\in\mathcal{O}^\times$ so $a\mathcal{O}^\times=b\mathcal{O}^\times$. Hence $\leq$ is antisymmetric. If $ab^{-1},bc^{-1}\in\mathcal{O}$, then $ac^{-1}=ab^{-1}bc^{-1}\in\mathcal{O}$, showing that $\leq$ is transitive. Take now $a,b\in F^\times$ such that $a\mathcal{O}^\times\leq b\mathcal{O}^\times$ and $c\in F^\times$. We have that $bc(ac)^{-1}=bcc^{-1}a^{-1}=ba^{-1}\in\mathcal{O}$, whence $ac\mathcal{O}^\times\leq bc\mathcal{O}^\times$. This shows that $\leq$ is compatible with the operation of $\Gamma$. Finally, $\leq$ is a total order since $\mathcal{O}$ is a valuation hyperring, so that $ab^{-1}\in\mathcal{O}$ or $ba^{-1}\in \mathcal{O}$ for all $a,b\in F^\times$.\par
We now show that $\pi$ is a valuation on $F$. Clearly, $\pi$ is a surjective map, onto the ordered abelian group $\Gamma$ with $\infty$ and (V1) holds. Since $\pi$ is a homomorphism of groups we obtain (V2). It remains to show that (V3) holds for $\pi$. Take $x,y\in F$. If one of them is $0_F$, then (V3) is straightforward. We may then assume that $x,y\in F^\times$ and that $x\mathcal{O}^\times\leq y\mathcal{O}^\times$. Take $z\in x+y$. We wish to show that $zx^{-1}\in\mathcal{O}$. By assumption we have that $yx^{-1}\in\mathcal{O}$, thus
\[
zx^{-1}\in(x+y)x^{-1}=1+yx^{-1}\subseteq\mathcal{O},
\]
where we used Lemma \ref{valstrict}.\par
Finally, we observe that, by definition
\[
\mathcal{O}_\pi=\{x\in F\mid \pi x\geq 1\}=\{x\in F\mid x\mathcal{O}^\times\geq 1\mathcal{O}^\times\}=\{x\in F\mid x\in \mathcal{O}\}=\mathcal{O}.
\]
\end{proof}

\begin{definition}
Let $(\Gamma_i,<_i)$ be (partially) ordered sets for $i=1,2$. A map $\sigma:\Gamma_1\to\Gamma_2$ is said to be \emph{order preserving} if $\gamma_1\leq_1\gamma_2$ implies that $\sigma(\gamma_1)\leq_2\sigma(\gamma_2)$ for all $\gamma_1,\gamma_2\in\Gamma_1$.
\end{definition}

\begin{definition}
For $i=1,2$ let $v_i:F\to\Gamma_i\cup\{\infty\}$ be valuations on a hyperfield $F$. We say that $v_1$ and $v_2$ are \emph{equivalent} if there exists an isomorphism of groups $\sigma:\Gamma_1\to\Gamma_2$ which is order preserving and such that $v_2=\sigma\circ v_1$. 
\end{definition}

\begin{lemma}\label{Valuegroupasquotient}
Let $v:F\to\Gamma\cup\{\infty\}$ be a valuation on a hyperfield $F$. Then $\Gamma\simeq F^{\times}/\mathcal{O}_v^\times$ with an isomorphism of groups which is order preserving.
\end{lemma}
\begin{proof}
We consider $F^{\times}/\mathcal{O}_v^\times$ as an ordered abelian group with the ordering defined in Proposition \ref{FtimesOtimes}. Using the surjectivity of $v$, we define a map
\[
\sigma:\Gamma\to F^{\times}/\mathcal{O}_v^\times
\]
by $\sigma(va)=a\mathcal{O}_v^\times$ for all $a\in F^\times$. This is well-defined since if $va=vb$, then $va-vb=v(ab^{-1})=0$ so that $ab^{-1}\in\mathcal{O}_v^\times$ and then $a\mathcal{O}_v^\times=b\mathcal{O}_v^\times$. Using property (V2) of valuations, we obtain that $\sigma$ is a homomorphism of groups. Further, if $va\leq vb$, then $ba^{-1}\in\mathcal{O}_v$ which means that $\sigma(va)\leq\sigma(vb)$. Thus, $\sigma$ is order preserving. It is clear that $\sigma$ is surjective. It therefore remains to show that $\sigma$ is injective. To this end, assume that $a\mathcal{O}_v^\times=b\mathcal{O}_v^\times$ for some $a,b\in F^\times$. Then there exists $c\in\mathcal{O}_v^\times$ such that $a=bc$. Since $vc=0$, by (V2) we obtain that $va=vb$. This completes the proof.
\end{proof}
\begin{remark}
Observe that by construction of $\sigma$ in the above proof, we have that $\sigma\circ v=\pi$ where $\pi$ is the canonical epimorphism $F^\times\to F^\times/\mathcal{O}_v^{\times}$.
\end{remark}

\begin{corollary}\label{eqvalhyp}
For $i=1,2$ let $v_i:F\to\Gamma_i\cup\{\infty\}$ be valuations on a hyperfield $F$. Then $v_1$ and $v_2$ are equivalent if and only if $\mathcal{O}_{v_1}=\mathcal{O}_{v_2}$.
\end{corollary}
\begin{proof}
By the previous lemma we obtain for $i=1,2$ that $\Gamma_i\simeq F^{\times}/\mathcal{O}_{v_i}^\times$ as ordered abelian groups with isomorphisms $\sigma_i$ such that $\sigma_i\circ v_i=\pi_i$ where $\pi_i:F^\times\to F^\times/\mathcal{O}_{v_i}^\times$ is the canonical projection for $i=1,2$. Thus, if $\mathcal{O}_{v_1}=\mathcal{O}_{v_2}$, then $\pi_1=\pi_2$ and $\sigma:=\sigma_2^{-1}\circ\sigma_1$ is an isomorphism of ordered abelian groups $\Gamma_1\to\Gamma_2$. Further we have that
\[
\sigma\circ v_1=\sigma_2^{-1}\circ(\sigma_1\circ v_1)=\sigma_2^{-1}\circ\pi_2=v_2.
\]
Hence, $v_1$ and $v_2$ are equivalent.\par
On the other hand, if $v_1$ and $v_2$ are equivalent, then we obtain that $F^{\times}/\mathcal{O}_{v_1}^\times\simeq F^{\times}/\mathcal{O}_{v_2}^\times$ as ordered abelian groups. In particular, for $a\in F^\times$ we have that $1\mathcal{O}_{v_1}^\times\leq a\mathcal{O}_{v_1}^\times$ if and only if $1\mathcal{O}_{v_2}^\times\leq a\mathcal{O}_{v_2}^\times$. Using the definition of the ordering in $F^{\times}/\mathcal{O}_{v_i}^\times$ we see that this means that $a\in\mathcal{O}_{v_1}$ if and only if $a\in\mathcal{O}_{v_2}$. Since $0\in\mathcal{O}_{v_i}$ for $i=1,2$, we conclude that $\mathcal{O}_{v_1}=\mathcal{O}_{v_2}$ as claimed.
\end{proof}
In the following sections we will always consider valuations on hyperfields up to equivalence.

\section{Krasner valued hyperfields}\label{sec3}

In this section, we focus our attention on those valued hyperfields which satisfy the original more restrictive axioms of Krasner and were called by him \textit{hypercorps valu\'e}. These valued hyperfields still attract the attention of mathematicians. For instance, they have recently been considered by Tolliver and Lee (see \cite{Tol16,Lee20}) who called them simply \lq\lq valued hyperfields\rq\rq.

\subsection{Ultrametric spaces}\label{UMS}

We begin by presenting some basic theory of ultrametric spaces, a notion as well studied by Krasner (cf.\ \cite{Kra44}). The axioms for an ultrametric distance can be formulated using just the linear order of non-negative real numbers where $0$ is a bottom element. Since nothing but the order is used from the structure of real numbers, we will use the term ultrametric space in a broader sense allowing ultrametric distances to take values in any linearly ordered set. In addition, since the value group of a valuation has a top element, our definition of ultrametric space below is a modification which better fits into our context (the same approach can be found e.g.\ in \cite{Kuh11}). In the case of real-valued distances, this modification corresponds to a replacement of the linearly ordered set $(\R_{\geq 0},<)$ with $(\R\cup\{\infty\},>)$. 

\begin{definition}
An \emph{ultrametric distance} (or simply an \emph{ultrametric}) on a set $X$ is a function $d:X\times X\to\Gamma\cup\{\infty\}$, where $(\Gamma,<)$ is a linearly ordered set and $\infty$ satisfies $\gamma<\infty$ for all $\gamma\in\Gamma$, such that for all $x,y,z\in X$
\begin{itemize}
\item[(U1)] $d(x,y)=\infty$ if and only if $x=y$,
\item[(U2)] $d(x,y)=d(y,x)$,
\item[(U3)] $d(x,z)\geq\min\{d(x,y),d(y,z)\}$.
\end{itemize}
We call $(X,d)$ an \emph{ultrametric space} whenever $d$ is an ultrametric on $X$. We call the set $dX:=\{d(x,y)\mid x,y\in X, x\neq y\}\subseteq \Gamma$ the \emph{value set} of $d$. 
\end{definition}

\begin{example}
Let $(K,v)$ be a valued field. Then the function
\begin{align*}
K\times K&\to\Gamma\cup\{\infty\}\\
(x,y)&\mapsto v(x-y)
\end{align*} 
is an ultrametric on $K$. 
\end{example}

\begin{definition}
Let $(X,d)$ be an ultrametric space. A subset $B\subseteq X$ is called a \emph{ball} if for all $y,z\in B$ and all $x\in X$ we have that the following implication
\[
d(x,y)\geq d(y,z)\quad\Longrightarrow\quad x\in B
\]
holds for all $x,y,z\in X$.
\end{definition}

\begin{definition}
A subset $\rho$ of a linearly ordered set $(\Gamma,<)$ is called an \emph{initial segment} (resp.\ \emph{final segment}) if for all $\delta\in\rho$ and all $\gamma\in\Gamma$ if $\gamma<\delta$ (resp.\ $\gamma>\delta$), then $\gamma\in\rho$.
\end{definition}

\begin{lemma}\label{ultra}
For every element $x$ of an ultrametric space $(X,d)$ and every final segment $\rho$ of $dX\cup\{\infty\}$, we have that
\[
B_\rho(x):=\{y\in X\mid d(x,y)\in \rho\}.
\] 
is a ball in $X$. Conversely, if $B$ is a ball in $X$ and $\rho$ is the smallest (with respect to inclusion) final segment of $dX\cup\{\infty\}$ containing $d(y,z)$ for all $y,z\in B$, then for every $x\in B$,
\[
B=B_\rho(x).
\]
In particular, $B_\rho(x)=B_\rho(y)$ for every $y\in B_\rho(x)$.
\end{lemma}

\begin{proof}
Assume that $y,z\in B_\rho(x)$, that is, $d(x,y)\in \rho$ and $d(x,z)\in \rho$. If $t\in X$ is such that $d(y,t)\geq d(y,z)$, then the inequalities
\[
d(x,t)\geq\min\{d(x,y),d(y,t)\}\geq\min\{d(x,y),d(y,z)\}\geq\min\{d(x,y),d(y,x),d(x,z)\}\in\rho
\] 
follow from (U2) and (U3). Since $\rho$ is a final segment of $dX\cup\{\infty\}$, we conclude that $d(x,t)\in \rho$. Hence, $t\in B_\rho(x)$and $B_\rho(x)$ is a ball.\par
For the converse, assume that $B$ is a ball and let $\rho$ be as in the assertion. Further, let $x$ be any element in $B$. If $y\in B$, then $d(x,y)\in \rho$ and thus $y\in B_\rho(x)$. On the other hand, if $y\in B_\rho(x)$, then $d(x,y)\in \rho$. So by definition of $\rho$, there is some $z\in B$ such that $d(x,z)\leq d(x,y)$. Since $B$ is a ball, it follows that $y\in B$. We have proved that $B=B_\rho(x)$. 
\end{proof}

\begin{corollary}\label{ultraballs}
Let $(X,d)$ be an ultrametric space. Every two balls with non-empty intersection are comparable by inclusion.
\end{corollary}

\begin{proof}
Take two balls $B$ and $B'$ and suppose that $z\in B\cap B'$. By Lemma \ref{ultra} there are final segments $\rho,\varsigma$ of $dX\cup\{\infty\}$ such that $B=B_\rho(z)$ and $B'=B_\varsigma(z)$. Since $\rho$ and $\varsigma$ are final segments, we must have $\rho\subseteq\varsigma$ or $\varsigma\subseteq\rho$. Hence, $B\subseteq B'$ or $B'\subseteq B$.
\end{proof}

\subsection{Krasner valuations}

Let $(\Gamma,<,+,0)$ be an ordered abelian group, $\rho$ be an initial segment of $\Gamma$ and $\gamma$ an element of $\Gamma$. We will sometimes write $\gamma>\rho$ to indicate that $\gamma\notin\rho$.\par 
The following class of valued hyperfields is of special interest.

\begin{definition}[Section 3 of \cite{Kra57}, Definition 1.4 in \cite{Tol16} and Definition 2.4 in \cite{Lee20}]\label{KVH}
We call a valued hyperfield $(F,v)$ a \emph{Krasner valued hyperfield} if
\begin{itemize}
\item[(KVH1)] For all $x,y\in F$, $v(x+y)$ is a singleton unless $0\in x+y$.
\item[(KVH2)] There exists an initial segment $\rho_v$ of $vF$ such that $0\in\rho_v$ and for all $x,y,z,t\in F$ we have that $z\in x+y$ implies that $t\in x+y$ if and only if $vs>\rho_v+\min\{vx,vy\}$ for all $s\in z-t$.
\end{itemize}
The initial segment $\rho_v$ is called the \emph{norm} of $v$. We will also say that $v$ is a \emph{Krasner valuation on $F$} when $(F,v)$ is a Krasner valued hyperfield. 
\end{definition}

\begin{example}\label{KrasK}
Let $(K,v)$ be a valued field and consider $K$ as a hyperfield. Then $v$ is a Krasner valuation on $K$ with norm $vK$.
\end{example}

\begin{proposition}\label{Kraultr}
Let $(F,v)$ be a Krasner valued hyperfield. For all $x,y\in F$ such that $x\neq y$, by axiom (KVH1), $v(x-y)$ contains a unique element $\gamma_{x,y}\in vF$. Define a map $d_v:F\times F\to\Gamma\cup\{\infty\}$ as
\[
d_v(x,y):=\begin{cases}\gamma_{x,y}&\text{if }x\neq y,\\ \infty&\text{otherwise.}\end{cases}
\]
Then $d_v$ is an ultrametric on $F$. Moreover, for all $x,y\in F$ and any $z\in x+y$, in the ultrametric space $(F,d_v)$ we have that
\[
x+y=B_{\rho_v+\min\{vx,vy\}}(z).
\]
\end{proposition}

\begin{proof}
Axiom (U1) follows from axiom (V1), axiom (U2) follows from Corollary \ref{valringFVK} $(ii)$ and axiom (U3) is a direct consequence of axiom (V3). The last statement is just a reformulation of axiom (KVH2). 
\end{proof}

We call the ultrametric $d_v$ on a Krasner valued hyperfield $(F,v)$ the \emph{ultrametric induced by }$v$ on $F$.

\begin{remark}\label{trivialKrasval}
Let $F$ be a hyperfield and let $v$ be the trivial valuation on $F$. If $v$ is a Krasner valuation of norm $\rho_v$, then by (KVH2) for all $x,y\in F^\times$ we have that $x\in 1-1$ if and only if $0=vx>\rho_v$. Since $0\in\rho_v$, it follows that $x-x=\{0\}$ for all $x\in F$ and then $F$ is a field by Lemma \ref{1-1=0}.\par
Conversely, if $K$ is a field, then the map $v(0):=\infty$ and $vx:=0$ for all $x\in K^\times$ is a Krasner valuation on $K$ with value group $vK=\{0\}$ and norm $vK$.
\end{remark}

Krasner's main motivation probably came from the following example.

\begin{example}\label{Kgamma}
For a valued field $(K,v)$ and an initial segment $\rho\subseteq vK$ containing $0$, let us consider the (multiplicative) group of \emph{$1$-units of level $\rho$}:
\[
1+\mathcal{M}^\rho_v=\{x\in K\mid v(x-1)>\rho\}\subseteq\mathcal{O}_v^\times.
\] 
Then $v_\rho:=v_{1+\mathcal{M}^\rho_v}$ is a Krasner valuation on $K_\rho:=K_{1+\mathcal{M}_v^\rho}$ and the norm of $v_\rho$ is $\rho$.
\end{example}

Actually, any Krasner valued hyperfield which is a factor hyperfield is of this form, as we show below.

\begin{proposition}
Let $F$ be a factor hyperfield admitting a Krasner valuation $w$. Then $F=K_\rho$ and $w$ is $v_\rho$ for some valued field $(K,v)$ and some initial segment $\rho$ of $vK=wF$ containing $0$.
\end{proposition}

\begin{proof}
By Lemma \ref{KTvalK} there is a valued field $(K,v)$ and a subgroup $T\subseteq\mathcal{O}_w^\times$ of $K^\times$ such that $v_T=w$. Suppose that $t\in T$ is not a $1$-unit in $K$. Then $vt=0$ and $v(t-1)=0$ must hold. now, on the one hand, since $w=v_T$ is a Krasner valuation, for all $[x]_T\in [1]_T-[1]_T$ we have that $vx=v_T[x]_T>0$ by (KVH2). On the other hand, since $t\in T$ we have that $[t-1]_T\in[1]_T-[1]_T$. This contradiction proves that $T\subseteq 1+\mathcal{M}_v$ must hold and the result follows.
\end{proof}

\begin{remark}
We do not know if all Krasner valued hyperfields are factor hyperfields. Some results connected to this problem are provided in \cite{LT22}.
\end{remark}

\begin{example}
Consider the Hahn series field $K:=\F_2((t^\Gamma))$ for some non-trivial ordered abelian group $\Gamma$ and let $v$ denote its canonical $t$-adic valuation. In this case, since $Kv\simeq\F_2$ we have that $\mathcal{O}_v^\times=1+\mathcal{M}_v$. We conclude that
\[
\mathcal{T}'(\Gamma)\simeq K_\rho,
\]
where $\rho=\{\gamma\in\Gamma\mid\gamma\leq 0\}$. We have that the identity map $v_\rho$ is the identity map on $\mathcal{T}(\Gamma)=\mathcal{T}'(\Gamma)$ and it is a Krasner valuation on $\mathcal{T}'(\Gamma)$ with norm $\rho$. Note that, the same map is \emph{not} a Krasner valuation on $\mathcal{T}(\Gamma)$ as $0\in [0,\infty]=0\boxplus 0$ violates (KVH2). 
\end{example}

We now study the residue hyperfield of a Krasner valued hyperfield.

\begin{proposition}\label{reshfld}
The residue hyperfield $Fv$ of a Krasner valued hyperfield $(F,v)$ is a field.
\end{proposition}

\begin{proof}
Take $xv\in 1v-1v$. This means that $x\in 1-1$ and by axiom (KVH2), we deduce that $vx>\rho_v$. Since $0\in\rho_v$, it follows that $xv=0v$ and therefore $Fv$ is a field by Lemma \ref{1-1=0}.
\end{proof}

\begin{example}
Let $(K,v)$ be a valued field. It follows from Proposition \ref{resquot} that, for any initial segment $\rho$ of $vK$ containing $0$, the residue hyperfield of $K_\rho$ is a field isomorphic to $Kv$.
\end{example}

Let us now consider a valued hyperfield which is not a Krasner valued hyperfield.

\begin{example}\label{noKraVal}
Consider the rational function field $K:=\Q(X)$ over the rationals, and its multiplicative subgroup $T:=\Q^\times$. Under the canonical $X$-adic valuation $v$, we have that $T\subseteq\mathcal{O}_{v}^\times$. Hence, Lemma \ref{vT} yields a valued hyperfield $(K_T,v_T)$. By Proposition \ref{resquot} and Proposition \ref{exquot} $(i)$, we have that 
\[
K_Tv_T\simeq\Q_{\Q^\times}\simeq\K.
\]
Since $\K$ is not a field, the Proposition \ref{reshfld} implies that $(K_T,v_T)$ is not a Krasner valued hyperfield.
\end{example}

\subsection{Mittas and superiorly canonical hypergroups}

J.\ Mittas was a student of Krasner. He introduced superiorly canonical hypergroups (defined below) and published (sometimes without proofs) some results which connect them to Krasner valuations (see e.g.\ \cite{Mit71}). Some of the contents of this subsection were inspired by his work.

\begin{definition}[\cite{Mit71} and page 81 of \cite{Mas21}]\label{supcanhypg}
A canonical hypergroup $H$ is called \emph{superiorly canonical} if
\begin{itemize}
\item[(SCH1)] For all $x,y\in H$, if $x\in x+y$, then $x+y=\{x\}$.
\item[(SCH2)] For all $x,y,z,t\in H$, if $(x+y)\cap(z+t)\neq\emptyset$, then $x+y\subseteq z+t$ or $z+t\subseteq x+y$. 
\item[(SCH3)] For all $x,y\in H$ such that $x\neq y$ we have that $z,t\in x-y$ implies $z-z=t-t$.
\item[(SCH4)] For all $x,y,z\in H$, if $x\in z-z$ and $y\notin z-z$, then $x-x\subseteq y-y$.
\end{itemize}
\end{definition}

\begin{example}
Any abelian group is a superiorly canonical hypergroup.
\end{example}

Other examples of superiorly canonical hypergroups are provided by the following result.

\begin{proposition}\label{Mittas1}
Let $(F,v)$ be a Krasner valued hyperfield. Then the additive canonical hypergroup of $F$ is superiorly canonical.
\end{proposition}
\begin{proof}
We verify the axioms one by one. Assume that $x\in x+y$ for some $x,y\in F$. By reversibility, $y\in x-x$, so the inequalities $vy>\rho_v+vx\geq vx$ follow from axiom (KVH2) and $0\in\rho_v$. Therefore, if $z\in x+y$, then $vx=vz$ by Corollary \ref{valringFVK} $(iv)$ and since by reversibility $y\in z-x$ and $d_v(y,0)=vy>\rho_v+\min\{vx,vz\}$, axiom (KVH2) implies that $0\in z-x$ and so $x=z$ must hold. This shows (SCH1).\par
Axiom (SCH2) follows from Proposition \ref{Kraultr} and Corollary \ref{ultraballs}.\par
For (SCH3), take $x,y\in F$ and assume that $x\neq y$, i.e., $0\notin x-y$. Take $z,t\in x-y$. We claim that $vz=vt$. Suppose that $vz<vt$, then by Corollary \ref{valringFVK} $(iv)$ we have that $va=vz$ for all $a\in z-t$, thus
\[
d_v(z,0)=vz=d_v(z,t)>\rho_v+\min\{vx,vy\}.
\]
Axiom (KVH2) now implies that $0\in x-y$, a contradiction. Therefore, $vz\geq vt$. Symmetrically, $vt\geq vz$ and so $vz=vt$ as claimed. Now, by (KVH2) we have that $a\in z-z$ if and only if $va>\rho_v+vz$ and $a\in t-t$ if and only if $va>\rho_v+vt$. Since $vz=vt$ we conclude that $z-z=t-t$. This shows that (SCH3) holds.\par
For (SCH4), take $x\in z-z$ and $y\notin z-z$. By axiom (KVH2) it follows that $vx>\rho_v+vz$ and $vy\in\rho_v+vz$. In particular, $vx>vy$. Now, again by axiom (KVH2), if $a\in x-x$, then $va>\rho_v+vx$ which implies $va>\rho_v+vy$ and so $a\in y-y$. Hence, $x-x\subseteq y-y$. This completes the proof.
\end{proof}

The theory of superiorly canonical hypergroups and the theory of Krasner valued hyperfields are even more deeply related.

\begin{proposition}\label{Mittas2}
Let $F$ be a hyperfield with a superiorly canonical additive hypergroup. Then
\[
\mathcal{O}:=\{x\in F\mid x-x\subseteq 1-1\}
\]
is a valuation hyperring in $F$. Moreover, if $v$ is the valuation such that $\mathcal{O}_v=\mathcal{O}$, then $(F,v)$ is a Krasner valued hyperfield.
\end{proposition}
\begin{proof}
If $F$ is a field, then $\mathcal{O}=F$ and thus $v$ is the trivial valuation on $F$. Since $F$ is a field, $(F,v)$ is a Krasner valued hyperfield. We can then assume that $F$ is not a field.\par
If $x-x\subseteq 1-1$ and $y-y\subseteq 1-1$, then
\[
xy-xy=(x-x)y\subseteq (1-1)y=y-y\subseteq 1-1.
\]
Hence, $\mathcal{O}$ is multiplicatively closed.\par
If $x-x\nsubseteq 1-1$, then $1-1\subsetneq x-x$ by axiom (SCH2), since $0\in(x-x)\cap(1-1)$. Multiplying by $x^{-1}$ we obtain that $x^{-1}-x^{-1}\subsetneq 1-1$. Therefore, $x^{-1}\in\mathcal{O}$.\par
Assume that $x-x\subseteq 1-1$ and $y-y\subseteq 1-1$. We claim that if $z\in x+y$, then $z-z\subseteq 1-1$. Pick $a\in z-z$. Since $z\in x+y$ we have that
\[
a\in z-z\subseteq (x-x)+(y-y)\subseteq (1-1)+(1-1),
\]
so there exists $x'\in 1-1$ and $y'\in 1-1$ such that $a\in x'+y'$. Hence, $1\in x'+1$ by reversibility, and
\[
a\in x'+y'\subseteq (x'+1)-1=1-1,
\]
where we have used axiom (SCH1). We proved that $\mathcal{O}$ is a valuation hyperring in $F$.\par
Clearly, $\mathcal{O}^\times=\{x\in F\mid x-x=1-1\}$. Set $vF:=F^\times/\mathcal{O}^\times$ and let $v:F^\times\to vF$ be the canonical epimorphism. We have to show that $(F,v)$ is a Krasner valued hyperfield. It is a valued hyperfield by Proposition \ref{FtimesOtimes}. Let us now verify the validity of axioms (KVH1) and (KVH2).\par
Take $x,y\in F^\times$ and assume that $0\notin x+y$. Pick $z,t\in x+y$ and suppose that $vz<vt$. This means that $tz^{-1}\in\mathcal{O}\setminus\mathcal{O}^\times$, so
\[
tz^{-1}-tz^{-1}\subsetneq 1-1.
\]
But then $t-t\subsetneq z-z$ which contradicts axiom (SCH3). We have proved that (KVH1) holds for $(F,v)$. \par
We now have to verify (KVH2). 
Our first claim is that $v(1-1)$ is a final segment of $vF$ which does not contain $0$. Pick $a\in 1-1$ and $\gamma>va$. Let $b\in F^\times$ be such that $vb=\gamma$. Since $vb>va$, we have that $b-b\subsetneq a-a$. Now, if $b\notin 1-1$, then axiom (SCH4) implies $a-a\subseteq b-b$. Therefore, $b\in1-1$ must hold and thus $\gamma=vb\in v(1-1)$ and $v(1-1)$ is a final segment.\par 
For $x\in F^\times$, if $x\in x-x$, then $x-x=\{x\}$ by axiom (SCH1). However, since $0\in x-x$, this cannot be. An element $x\in F^\times$ has value $0$ if and only if $x-x=1-1$. Hence it cannot belong to $1-1$ as $x\notin x-x$ for all $x\in F^\times$. This shows that $v(1-1)$ does not contain $0$.\par
We let $\rho_v$ be the complement in $vF$ of $v(1-1)$. Then $\rho_v$ is an initial segment of $vF$ which contains $0$.\par 
Observe that for all $x\in F$ we have that $a\in x-x=(1-1)x$ if and only if there exists $y\in 1-1$ such that $a=yx$. This implies that $va=vy+vx>\rho_v+vx$. Conversely, if $va>\rho_v+vx$, then $v(ax^{-1})\in v(1-1)$, so there exists $b\in 1-1$ such that $vb=v(ax^{-1})$. Therefore, 
\[
b\in 1-1=bxa^{-1}-bxa^{-1}=b(xa^{-1}-xa^{-1}),
\]
so $1\in xa^{-1}-xa^{-1}$ implying that $a\in x-x$.\par
Take now $x,y\in F$ such that $x\neq y$ and $vx\leq vy$. Fix $z\in x-y$. We have to show that $t\in x-y$ if and only if $d_v(z,t)>\rho_v+vx$.\par 
Assume first that $t\in x-y$. If $t=z$ there is nothing to show. Otherwise, it suffices to show that $z-t\subseteq x-x$ by what we have already shown above. Take $a\in z-t$. Since $z\in x-y$, we have that
\[
a\in z-t\subseteq x-(y+t).
\]
Hence, there exists $b\in y+t$ such that $a\in x-b$. We obtain that $b\in (x-a)\cap(y+t)$, so $y+t\subseteq x-a$ or $x-a\subseteq y+t$ by axiom (SCH2). In the first case, since $x\in y+t$ we have that $x\in x-a$ and $a\in x-x$ follows. It remains to deal with the case $x-a\subsetneq y+t$. In this case, since $t\in x-y$ and $y-y\subseteq x-x$, we obtain that
\[
x-a\subsetneq y+t\subseteq y+x-y\subseteq x+(x-x).
\]
Take $c\in x-a$. There exists $x'\in x-x$ such that $c\in x+x'$. This implies that
\[
a\in x-c\subseteq x-(x+x')=x-x
\]
where we have used axiom (CH4) and axiom (SCH1) to conclude that $x+x'=\{x\}$.\par
Now, assume that $d_v(z,t)>\rho_v+vx$. We have to prove that $t\in x-y$. If $z=t$, then there is nothing to show. Hence, we can assume that $z\neq t$ and so $z-t\subsetneq x-x$ must hold. We have that
\[
t\in t+0\subseteq t+z-z\subsetneq x-x+z.
\]
Hence, there is $b\in z-x$ such that $t\in x+b$. We conclude that $b\in(z-x)\cap(t-x)$. By axiom (SCH2) we then obtain that $z-x\subseteq t-x$ or $t-x\subseteq z-x$. In the first case, $-y\in z-x\subseteq t-x$ and $t\in x-y$ follows. It remains to deal with the case $t-x\subsetneq z-x$. In this case, since $z\in x-y$, we have that
\[
t-x\subsetneq z-x\subseteq x-y-x=x-x-y.
\]
Take $a\in t-x$. There exists $x'\in x-x$ such that $a\in x'-y$. Now, by reversibility and since $a\in t-x$, we have that 
\[
-y\in a-x'\subseteq t-(x+x')=t-x,
\]
where for the last equality we have used axiom (SCH1). Now, $t\in x-y$ follows by axiom (CH4).
\end{proof}
Directly from Proposition \ref{Mittas1} and Proposition \ref{Mittas2} above, we deduce the following characterization theorem for the hyperfields which admit a Krasner valuation.
\begin{theorem}\label{maintw}
Let $F$ be a hyperfield. Then $F$ admits a Krasner valuation if and only if the additive hypergroup of $F$ is superiorly canonical.
\end{theorem}
\begin{example}
The additive hypergroup of the hyperfield that we have considered in Example \ref{noKraVal} above is not superiorly canonical. Indeed,
\[
1T+1T=\{0T,1T\}
\]
and thus (SCH1) fails. It follows from Theorem \ref{maintw} that this hyperfield does not admit Krasner valuations at all.
\end{example}
\begin{example}
Consider a generalised tropical hyperfield $\mathcal{T}(\Gamma)$, where $\Gamma$ is some non-trivial ordered abelian group (see Example \ref{TGamma}). By Proposition \ref{exquot} $(iv)$ and Lemma \ref{vT} we conclude that $\mathcal{T}(\Gamma)$ is a valued hyperfield (see also the discussion after Proposition \ref{Coars}). Nevertheless, by Theorem \ref{maintw}, $\mathcal{T}(\Gamma)$ does not admit Krasner valuations as its additive hypergroup does not satisfy (SCH1).
\end{example}
Let us now analyse another example.
\begin{example}\label{last}
Take $K=\Q(X)$ with the valuation $v:=v_p\circ v_X$ where $v_X$ denotes the $X$-adic valuation on $\Q(X)$ and $v_p$ denotes the $p$-adic valuation on $\Q$, for some prime number $p$. More explicitly, for a polynomial $f(X)=\sum_{i=0}^na_iX^i$ in $\Q[X]$, we have
\[
vf(X)=\left(v_Xf(X),v_p(a_{v_Xf(X)})\right)\in\mathbb{Z}\times\mathbb{Z}.
\]
The order relation in $vK$ is the lexicographic order of $\mathbb{Z}\times\mathbb{Z}$, that is, $(n_1,n_2)<(m_1,m_2)$ if and only if $n_1<m_1\vee (n_1=m_1\wedge n_2<m_2)$.\par
Consider the Krasner valued hyperfield $F:=K_\rho$ where $\rho$ is the smallest initial segment of $\mathbb{Z}\times\mathbb{Z}$ which contains $\{0\}\times\mathbb{Z}$ (cf.\ Example \ref{Kgamma}). Let us denote its Krasner valuation $v_\rho$ by $w$.\par
Note that, if $(m_1,m_2)\in\rho$ and $m_1>0$, then, since $\rho$ is an initial segment, we have that 
\[
\{0\}\times\mathbb{Z}\subseteq\{(k_1,k_2)\in\mathbb{Z}\times\mathbb{Z}\mid(k_1,k_2)\leq(m_1-1,m_2)\}\subsetneq\rho,
\]
in contradiction with the minimality of $\rho$. Therefore, $m_1\leq 0$ for all $(m_1,m_2)\in\rho$. Conversely, if $m_1\leq 0$, then $(m_1,m_2)\leq(0,m_2)$ for all $m_2\in\mathbb{Z}$ and therefore $(m_1,m_2)\in\rho$ since $\rho$ is an initial segment containing $\{0\}\times\mathbb{Z}$. It follows that
\begin{equation}\label{rho}
\rho=\{(m_1,m_2)\in\mathbb{Z}\times\mathbb{Z}\mid m_1\leq 0\}=\{(m_1,m_2+n)\in\mathbb{Z}\times\mathbb{Z}\mid m_1\leq 0\}=\rho+(0,n)
\end{equation}
for any $(0,n)\in\{0\}\times\mathbb{Z}$.\par
By Proposition \ref{Mittas1}, the additive hypergroup of $F$ is superiorly canonical. Therefore, by Proposition \ref{Mittas2}, we have the valuation hyperring
\[
\mathcal{O}_u=\{x\in F\mid x-x\subseteq 1-1\}
\]
for some Krasner valuation $u$ on $F$. Let us now show that $w$ and $u$ are not equivalent valuations on $F$.\par
In fact, we will show that $\mathcal{O}_w\subsetneq\mathcal{O}_u$. Pick $x\in \mathcal{O}_w$, i.e., $wx\geq (0,0)$. By axiom (KVH2) we have that $y\in x-x$ if and only if $wy>\rho+wx\geq\rho$. Another application of axiom (KVH2) shows that $x-x\subseteq 1-1$ so that $x\in\mathcal{O}_u$. Now, consider e.g. the element $x$ of $F$ corresponding to the rational number $p^{-1}$ in $K$. Since $v_p(p^{-1})=-1$ we have that $wx=(0,-1)<(0,0)$ so that $x\notin\mathcal{O}_w$. On the other hand, since $wx=(0,-1)\in\{0\}\times\mathbb{Z}$, we have that $\rho+wx=\rho$ by \eqref{rho} and thus using (KVH2), we obtain that $y\in x-x$ if and only if $wy>\rho$ if and only if $y\in1-1$. This implies that $x-x\subseteq 1-1$ and thus $x\in\mathcal{O}_u$.
\end{example}

The above example constitutes the starting point for the discussion that we will have in Section \ref{sec5}.

\section{More on ordered abelian groups}\label{sec4}

In this section, we characterise generalised tropical hyperfields and discuss the concept of coarsening of a valuation in the multivalued setting. We derive some conclusions on the relation between ordered abelian groups and generalised tropical hyperfields.

\subsection{Characterisation of generalised tropical hyperfields}

To characterise generalised tropical hyperfields, we will apply another remarkable characterisation result. A hyperfield is \emph{stringent} if $x+y$ is a singleton whenever $0\notin x+y$. Bowler and Su in \cite{BS21} characterised stringent hyperfields and we will now briefly recall their result.\par 
In \cite[Section 4]{BS21} a natural construction of a hyperfield arising from a short exact sequence
\begin{equation}\label{ses}
\begin{tikzcd}
\{1\}\arrow[r]&F^\times\arrow[r,"\varphi"] & H^\times\arrow[r,"\psi"]&\Gamma\arrow[r]&\{0\}
\end{tikzcd}
\end{equation}
where $\Gamma$ is an ordered abelian group and $F$ is any hyperfield is described. The thus obtained hyperfield has $H^\times$ as multiplicative group and is called the \emph{$\Gamma$-layering of} $F$ \emph{along the short exact sequence} \eqref{ses}. After that the following theorem is proved.
\begin{theorem}[Theorem 4.10 in \cite{BS21}]\label{BS}
A hyperfield $H$ is stringent if and only if it is the \emph{$\Gamma$-layering of} $F$ \emph{along the short exact sequence} \eqref{ses} with $F$ being (isomorphic to) either $\K$, $\mathbb{S}$ or a field.
\end{theorem}

From the details of the construction (which we omit for brevity) it is not difficult to verify that the extension of the map $\varphi$ sending $0_F$ to $0_H$ is in any case an embedding of hyperfields.

We say that a hyperfield has characteristic $2$ if $0\in 1+1$ and that it has C-characteristic $1$ if $1\in 1+1$. For example, $\mathbb{S}$ satisfies the latter but not the former while $\K$ satisfies both. For more information on characteristic and C-characteristic of hyperfields the reader can see \cite{LKS23}.

We are now ready to state and prove our characterisation theorem.

\begin{theorem}[Characterisation of generalised tropical hyperfields]\label{charTGamma}
Generalised tropical hyperfields are precisely stringent hyperfields of characteristic $2$ and C-characteristic $1$.
\end{theorem}

\begin{proof}
The reader can easily verify from the relevant definitions that a generalised tropical hyperfield $\mathcal{T}(\Gamma)$ is a stringent hyperfield of characteristic $2$ and C-characteristic $1$.\par
For the converse, let $H$ be a stringent hyperfield of characteristic $2$ and C-characteristic $1$. By Theorem \ref{BS} and our observations on the extension of the map $\varphi$ above, we have that either $\K$, $\mathbb{S}$ or a field embed into $H$. Since $H$ has characteristic $2$ we have that $1=-1$, thus $\mathbb{S}$ cannot embed into $H$. On the other hand, any field cannot embed into $H$ neither. Indeed, our assumption $0,1\in 1+1$ implies that $\{0,1\}=(1+1)\cap\{0,1\}\subseteq (1+1)\cap F=1+_F 1$ is not a singleton. It follows that $H$ is the $\Gamma$-layering of $\K$ along a short exact sequence 
\[
\begin{tikzcd}
\{1\}\arrow[r]&\K^\times\arrow[r,"\varphi"] & H^\times\arrow[r,"\psi"]&\Gamma\arrow[r]&\{0\}
\end{tikzcd}
\]
for some ordered abelian group $\Gamma$. Now, since $\K^\times=\{1\}$, we obtain that the multiplicative group of $H$ is isomorphic to $\Gamma$ and the details of the construction given in \cite[Section 4]{BS21} show that the hyperoperation of $H$ is the one of $\mathcal{T}(\Gamma)$. We conclude that $H$ is a generalised tropical hyperfield. 
\end{proof}

An analogous reasoning yields to the following (but maybe less interesting) characterisation of strict generalised tropical hyperfields.

\begin{theorem}
Generalised tropical hyperfields are precisely the $\Gamma$-layerings of $F$ along the short exact sequence \eqref{ses}, with $F=\F_2$.
\end{theorem}

\subsection{Coarsenings}

\begin{definition}
Let $(\Gamma,<,+,0)$ be an ordered abelian group. A subgroup $\Delta$ of $\Gamma$ is \emph{convex} if for all $\gamma\in\Gamma$ we have that if there exist $\delta_1,\delta_2\in\Delta$ such that $\delta_1<\gamma<\delta_2$, then $\gamma\in\Delta$.
\end{definition}

It is not difficult to see that the intersection of a family of convex subgroups of an ordered abelian group is again a convex subgroup and that the collection of all convex subgroups of an ordered abelian group is linearly ordered by inclusion. Let us recall another basic fact which makes convex subgroups important.

\begin{fact}
Let $(\Gamma,<,+,0)$ be an ordered abelian group and $\Delta$ a convex subgroup of $\Gamma$. Then $(\Gamma/\Delta,\prec,+,0)$ is an ordered abelian group, where
\[
x+\Delta\prec y+\Delta\iff x<y\text{ and }y-x\notin\Delta.
\]
In particular, the canonical epimorphism $\Gamma\to\Gamma/\Delta$ is order preserving.
\end{fact}

The following notion will play a fundamental role later.

\begin{definition}
Let $(\Gamma,<,+,0)$ be an ordered abelian group and $\rho$ an initial segment of $\Gamma$. We call the set
\[
\ig(\rho):=\{\gamma\in\Gamma\mid\rho+\gamma=\rho\}
\]
the \emph{invariance group} of $\rho$. If $\ig(\rho)=\{0\}$, then we say that $\rho$ has \emph{trivial invariance group}.
\end{definition}

\begin{remark}
In \cite{Kuh20,KuhCut} F.-V.\ Kuhlmann proves several results about invariance groups associated to initial segments in ordered abelian groups. In particular, it follows from \cite[Lemma 2.3]{KuhCut} that the invariance group of an initial segment $\rho$ of an ordered abelian group $\Gamma$ such that $0\in\rho$ is a convex subgroup of $\Gamma$ contained in $\rho$. 
\end{remark}

\begin{definition}
Let $v,w$ be two valuations on a hyperfield $F$. We say that $w$ is a \emph{coarsening} of $v$ if $\mathcal{O}_v\subseteq\mathcal{O}_w$. 
\end{definition}

In classical valuation theory for fields, it is well-known that to any convex subgroup of the value group one can associate a coarsening.
We note that this result generalises to valued hyperfields, actually with a much more conceptual proof.

\begin{proposition}\label{Coars}
Let $(F,v)$ be a valued hyperfield and let $\Delta$ be a convex subgroup of $vF$. Then 
\begin{align*}
v_\Delta:F&\to (vF/\Delta)\cup\{\infty\}\\
x&\mapsto vx+\Delta
\end{align*}
is a valuation on $F$ which is a coarsening of $v$.
\end{proposition}

\begin{proof}
By Lemma \ref{valhom}, $v:F\to\mathcal{T}(vF)$ is a surjective homomorphism of hyperfields. In addition, the canonical epimorphism $vF\to vF/\Delta$ extends to a surjective map $\mathcal{T}(vF)\to\mathcal{T}(vF/\Delta)$. From the fact that the latter is order preserving, it easily follows that it is a homomorphism of hyperfields. We conclude that the composition $v_\Delta:F\to\mathcal{T}(vF/\Delta)$ of $v$ with the latter map is a surjective homomorphism of hyperfields. This proofs the first part of the statement. Now, since
\begin{align*}
\mathcal{O}_{v_\Delta}&=\{x\in F\mid v_\Delta x\succeq 0_{vF/\Delta}\}\\
&=\{x\in F\mid vx+\Delta\succeq\Delta\}\\
&=\{x\in F\mid vx\geq 0\text{ or }vx\in\Delta\}\\
&\supseteq\{x\in F\mid vx\geq 0\}=\mathcal{O}_v,
\end{align*}
we conclude that $v_\Delta$ is a coarsening of $v$.
\end{proof}

In the above proof, we have seen that if $\Delta$ is a convex subgroup of an ordered abelian group $\Gamma$, then the natural map
\[
\pi_\Delta:\mathcal{T}(\Gamma)\to\mathcal{T}(\Gamma/\Delta)
\]
is a surjective homomorphism of hyperfields, which, by Lemma \ref{valhom}, is a valuation on $\mathcal{T}(\Gamma)$ with value group $\Gamma/\Delta$. The valuation hyperring of this valuation is
\[
\mathcal{O}_\Delta=\{\gamma\in\Gamma\mid\gamma+\Delta\succeq0+\Delta\}=\{\gamma\in\Gamma\mid \gamma\in\Delta\text{ or }\gamma>\Delta\}
\] 
and its unique maximal hypeideal is
\[
\mathcal{M}_\Delta=\{\gamma\in\Gamma\mid\gamma>\Delta\}
\]
It follows that the residue hyperfield $\mathcal{T}(\Gamma)\pi_\Delta=\mathcal{O}_\Delta/\mathcal{M}_\Delta$ of $\mathcal{T}(\Gamma)$ with respect to the valuation $\pi_\Delta$ is isomorphic to $\mathcal{T}(\Delta)$. An isomorphism is given by the map
\begin{align*}
\mathcal{T}(\Delta)&\to\mathcal{T}(\Gamma)\pi_\Delta\\
\infty&\mapsto\infty\\
\delta&\mapsto\delta\pi_\Delta
\end{align*}

In the next result we show that the valuations of $\mathcal{T}(\Gamma)$ are (up to equivalence) all of this form.

\begin{proposition}
Let $\Gamma$ be an ordered abelian group and assume that $v:\mathcal{T}(\Gamma)\to\mathcal{T}(\Gamma')$ is a valuation. Then there exists a convex subgroup $\Delta$ of $\Gamma$ such that $\Gamma/\Delta\simeq\Gamma'$ via an order preserving isomorphism.
\end{proposition}

\begin{proof}
Let $0'\in\Gamma'$ denote the neutral element of the group $\Gamma'$, $\boxplus$ the hyperoperation of $\mathcal{T}(\Gamma)$ and $\boxplus'$ the hyperoperation of $\mathcal{T}(\Gamma')$. Set $\Delta:=v^{-1}(0')$. Since $v$ is a homomorphism of groups $\Gamma\to\Gamma'$, we have that $\Delta$ is a subgroup of $\Gamma$. Moreover, since $v$ is surjective, we have that $\Gamma/\Delta\simeq\Gamma'$ as groups by the first homomorphism theorem for groups. Assume that $\delta_1,\delta_2\in\Delta$ and that $\gamma\in\Gamma$ is such that $\delta_1\leq\gamma\leq\delta_2$. Hence, $\delta_2\in\gamma\boxplus\gamma$ and so $0'=v(\delta_2)\in v(\gamma)\boxplus' v(\gamma)$ since $v$ is a homomorphism of hyperfields. On the other hand, $\gamma\in\delta_1\boxplus\delta_1$ so that $v(\gamma)\in0'\boxplus'0'=[0',\infty]$. If $v(\gamma)>0'$, then $0'\notin [v(\gamma),\infty]=v(\gamma)\boxplus'v(\gamma)$, a contradiction. We conclude that $v(\gamma)=0'$ must hold and thus $\gamma\in\Delta$. We have shown that $\Delta$ is a convex subgroup of $\Gamma$. It remains to show that the isomorphism of groups $\sigma:\gamma+\Delta\mapsto v(\gamma)$ is order preserving. Assume that $\gamma_1+\Delta\prec\gamma_2+\Delta$. This means that $\gamma_1<\gamma_2$ and $\gamma_2-\gamma_1\notin\Delta$ hold in $\Gamma$. Thus, $\{\gamma_1\}=\gamma_1\boxplus\gamma_2$ and then $v(\gamma_1)\in v(\gamma_1)\boxplus' v(\gamma_2)$. Since $\sigma$ is bijective and $\gamma_1\neq\gamma_2$, we have that $v(\gamma_1)\neq v(\gamma_2)$, so $v(\gamma_1)\in v(\gamma_1)\boxplus' v(\gamma_2)=\bigl\{\min\{v(\gamma_1),v(\gamma_2)\}\bigr\}$. This shows that $v(\gamma_1)<v(\gamma_2)$ must hold in $\Gamma'$ and thus $\sigma$ is order preserving and the proof is complete. 
\end{proof}

The next result now follows from Corollary \ref{eqvalhyp}.

\begin{corollary}
Let $\Gamma$ be an ordered abelian group. There is a bijective correspondence between convex subgroups of $\Gamma$ and valuation hyperrings of $\mathcal{T}(\Gamma)$.
\end{corollary}

\section{Krasner valuations induced by the additive structure}\label{sec5}

In this section we apply some of the results obtained in the previous section to Krasner valued hyperfields. We begin by observing that coarsenings of Krasner valuations might not be Krasner valuations.

\begin{example}
Let $\Gamma$ be a non-trivial ordered abelian group with a non-trivial convex subgroup $\Delta$ and consider the identity map as a Krasner valuation $v:\mathcal{T}'(\Gamma)\to\mathcal{T}(\Gamma)$. As a map, the coarsening
\[
v_\Delta:\mathcal{T}'(\Gamma)\to\mathcal{T}(\Gamma/\Delta)
\]
is just the canonical epimorphism $\Gamma\to\Gamma/\Delta$ extended to send $\infty$ to $\infty$. Suppose that $v_\Delta$ is a Krasner valuation of norm $\rho$ for some initial segment $\rho$ of $\Gamma/\Delta$ containing $0_{\Gamma/\Delta}$. Then by (KVH2) for all $\gamma\in\Gamma$ we would have that
\[
\gamma>0\iff \gamma\in0\boxplus'0\iff \gamma+\Delta>\rho+(0+\Delta).
\]
On the other hand, since $0_{\Gamma/\Delta}\in\rho$ , any positive $\gamma\in\Delta$ would violate this equivalence.
\end{example}

Nevertheless, we have the following result.

\begin{theorem}\label{coarseningst}
Let $(F,v)$ be a Krasner valued hyperfield and let $w$ be the Krasner valuation on $F$ determined by
\[
\mathcal{O}_w=\{x\in F\mid x-x\subseteq 1-1\}.
\]
Then $w$ is equivalent to the coarsening of $v$ corresponding to $\ig(\rho_v)$. In particular, the coarsening of a Krasner valuation $v$ corresponding to $\ig(\rho_v)$ is a Krasner valuation.
\end{theorem}

\begin{proof}
We show that these two valuations have the same valuation hyperring. Indeed, for all $x\in F$ we have that $vx+\ig(\rho_v)\succeq 0_{vF/\ig(\rho_v)}$ if and only if $vx\in\ig(\rho_v)$ or $vx>\ig(\rho_v)$. In both cases by (KVH2) applied to $v$ we have that
\[
y\in x-x\iff vy>\rho_v+vx\supseteq\rho_v\Longrightarrow y\in 1-1,
\]
i.e., $y\in\mathcal{O}_w$. Conversely, $vx+\ig(\rho_v)\prec 0_{vF/\ig(\rho_v)}$ if and only if $vx\notin\ig(\rho_v)$ and $vx<0$. Therefore, $\rho_v+vx\subsetneq\rho_v$ and there exists $y\in F$ such that $vy\in\rho_v$ and $vy\notin\rho_v+vx$. By (KVH2) applied to $v$, the former implies $y\notin 1-1$ while the latter is equivalent to $y\in x-x$. We conclude that $x\notin\mathcal{O}_w$. We have proved that $v_{\ig(\rho_v)}$ and $w$ have the same valuation hyperring in $F$. The theorem follows from Corollary \ref{eqvalhyp}.
\end{proof}

\begin{corollary}
Let $F$ be a hyperfield admitting two Krasner valuations $v_1$ and $v_2$ of norm $\rho_1$ and $\rho_2$, respectively. Then the coarsening of $v_1$ corresponding to $\ig(\rho_1)$ is equivalent to the coarsening of $v_2$ corresponding to $\ig(\rho_2)$.
\end{corollary}

\begin{corollary}\label{Mittas11}
Let $(F,v)$ be a Krasner valued hyperfield. If $\rho_v$ has trivial invariance group, then 
\[
\mathcal{O}_v=\{x\in F\mid x-x\subseteq 1-1\}.
\]
\end{corollary}

\begin{corollary}\label{Mittas22}
Let $F$ be a hyperfield with a superiorly canonical additive hypergroup. Then the norm of the Krasner valuation $v$ determined on $F$ by the valuation hyperring
\[
\mathcal{O}_v=\{x\in F\mid x-x\subseteq 1-1\}
\]
has trivial invariance group.
\end{corollary}

\begin{corollary}\label{Mittas3}
Let $F$ be a hyperfield with a superiorly canonical additive hypergroup. There is a unique (up to equivalence) Krasner valuation $v$ on $F$ such that $\ig(\rho_v)=\{0\}$.
\end{corollary}

\begin{remark}
In the model theory of valued fields, the RV-structure (mentioned in the introduction) of level $\gamma$ of a valued field $(K,v)$, where $\gamma$ is a non-negative element of $vK$, is essentially the Krasner valued hyperfield\footnote{This hyperfield is called the valued $\gamma$-hyperfield in \cite{Lee20}} $K_\gamma:=K_{\rho_\gamma}$,  where $\rho_\gamma:=\{\delta\in vK\mid\delta\leq\gamma\}$. Since $\ig(\rho_\gamma)=\{0\}$, it follows, as a consequence of Corollary \ref{Mittas11}, that the valuation hyperring of the Krasner valuation on $K_{\gamma}$ induced by $v$ is always \emph{definable} in the language of hyperfields, e.g., using the ternary relation $z\in x+y$ to encode the multivalued operation $+$.
\end{remark}

\section{Further research}\label{sec6}

There are at least three interesting points that have been touched but not developed further in this manuscript. Below we briefly describe them.

\begin{enumerate}
\item In the multivalued setting, $(F,v)$, $vF$ and $Fv$ can all be described as (valued) hyperfields and they are not distinct structures. This applies in particular when $F$ is a field.

\item In the literature there are not many examples of infinite hyperfields that are not factor hyperfields. Essentially, only the work of Massouros \cite{Mas85, Mas1985} provides such examples. In all those examples $M$ we have by definition $x+x=M\setminus\{x\}$ for all $x\in M^\times$. If $v$ is a valuation on $M$ and $x\neq 1$, then $x,x^{-1}\in 1+1$, so by (V3) and Corollary \ref{valringFVK} $(iii)$ we have that $vx=0$ must hold for all $x\in M^\times$. It follows that $M$ only admits the trivial valuation. Is it true that any hyperfield admitting a non-trivial valuation is a factor hyperfield? 

\item In view of Theorem \ref{charTGamma}, a non-trivial valuation on a hyperfield $F$ is a surjective homomorphism onto a hyperfield $H$ satisfying the following three conditions:
\begin{itemize}
\item $H$ has characteristic $2$;
\item $H$ has C-characteristic $1$;
\item $H$ is stringent.
\end{itemize} 
From this point of view, the image of a valuation is a quite special hyperfield. For example, one could think that responsible for the fact that finite (hyper)fields only admit the trivial valuation is the third property of $H$, since there are many examples of finite hyperfields satisfying $0,1\in 1+1$. We speculate that investigating weakenings of the stringency property for $H$ may lead to a notion of \emph{generalised valuation} which have the potential to be applied also in the classical singlevalued setting e.g.\ to use (generalised) valuation-theoretic methods over finite fields.
\end{enumerate}

\bibliography{Linzivalhyp}
\bibliographystyle{abbrv}

\end{document}